\documentclass[10pt,a4paper,final,oneside,openright]{amsart} 

\usepackage{amsfonts}
\usepackage{amsmath}
\usepackage{amssymb}
\usepackage[applemac]{inputenc}
\usepackage{url}
\usepackage{hyperref}
\usepackage{color}

\usepackage{lipsum}
\usepackage[title]{appendix}
\usepackage{mathtools}
\usepackage{ulem}

\usepackage{graphicx}
\usepackage[T1]{fontenc}
\usepackage[english]{babel}
\usepackage[all]{xy}
\usepackage{enumitem}
\usepackage{amssymb,amsmath,amscd,amsfonts,amsthm,mathrsfs}
\usepackage{times}  
\usepackage{fancyhdr}
\usepackage{pinlabel}

\usepackage{hyperref}
\hypersetup{
	colorlinks=true,
	linkcolor=blue,
	filecolor=magenta,      
	urlcolor=cyan,
}

\newtheorem{theorem}{{Theorem}}[section]
\newtheorem{proposition}[theorem]{{Proposition}}
\newtheorem{isom.ext}[theorem]{{Trivial isometric extension}}

\newtheorem{definition}[theorem]{{Definition}}
\newtheorem{lemma}[theorem]{{Lemma}}
\newtheorem{corollary}[theorem]{{Corollary}}
\newtheorem{fact}[theorem]{{Fact}}
\newtheorem{remark}[theorem]{{Remark}}

\newtheorem{observation}[theorem]{{Observation}}

\newtheorem{example}[theorem]{{Example}}

\definecolor{greenbf}{rgb}{0, 0.7 ,0.3}
 \hypersetup{
colorlinks=true, breaklinks=true, urlcolor= blue, linkcolor= red,
citecolor= {greenbf}, }

\def\R{\mathbb{R}}

\def\Z{\mathbb{Z}}

\def\N{\mathbb{N}}
\def\F{\mathcal{F}}

\def\K{\mathcal K}
\def\G{\mathcal G}

\def\H{\mathcal H}

\def\Iso{\sf{Iso}}

\def\Aff{\sf{Aff}}

\def\GL{{\sf{GL}}}
\def\O{{\sf{O}}}
\def\SO{{\sf{SO}}}

\def\SL{{\sf{SL}}}

\def\Iso{{\sf{Iso}} }

\def\Mink{{\sf{Mink}} }

\def\det{{\sf{det}}}

\def\g{{\mathfrak{g}}}

\def\h{{\mathfrak{h}}}
\def\i{{\mathfrak{i}}}
\def\j{{\mathfrak{j}}}
\def\z{{\mathfrak{z}}}
\def\z{{\mathfrak{z}}}

\def\dd{{\displaystyle{d}}}

\definecolor{purple}{rgb}{0.65,0.12,0.94}
\definecolor{marron}{rgb}{0.85,0.26,0.70}
 
\begin{document}
	\title[Brinkmann Spacetimes]{On Completeness and dynamics of compact  Brinkmann Spacetimes
	}

	\author [L. Mehidi]{Lilia Mehidi}

	\address{Departamento de Geometria y Topologia\hfill\break\indent
Facultad de Ciencias, Universidad de Granada, Spain\\}
	\email{lilia.mehidi@ugr.es
    \hfill\break\indent
	\url{https://mehidi.pages.math.cnrs.fr/siteweb/}}

	\author[A. Zeghib]{Abdelghani Zeghib }
	\address{UMPA, CNRS, 
		\'Ecole Normale Sup\'erieure de Lyon\hfill\break\indent
		46, all\'ee d'Italie
		69364 LYON Cedex 07, FRANCE}
	\email{abdelghani.zeghib@ens-lyon.fr 
		\hfill\break\indent
		\url{http://www.umpa.ens-lyon.fr/~zeghib/}}
	\date{\today, MSC classes: 53C50, 53B30}
	\maketitle

	\begin{abstract}  Brinkmann Lorentzian manifolds are those admitting a lightlike parallel vector field.  We prove geodesic completeness of the compact and also compactly  Brinkmann-homogeneous spacetimes. We also prove, partially,  that their parallel 
		vector field generates an equicontinuous flow.

	\end{abstract}

	\tableofcontents

	\section{Introduction}
	

	
	There is a lack of completeness and a lack of compactness 
	of Lorentzian structures compared to Riemannian ones. Compact Riemannian manifolds
	are complete and have a compact isometry group, but compact Lorentzian  manifolds are 
	(generically) incomplete and some have a non-compact isometry group. 
	This is mainly  due to existence of 
	degenerate objects for Lorentzian metrics. A Brinkmann spacetime is a Lorentzian manifold admitting a lightlike parallel vector field $V$ (see Definition \ref{Definition: Brinkmann spacetime}). So Brinkmann Lorentzian manifolds appear 
	as a consecration of this phenomenon. Moreover, being parallel, $V$ is also a Killing vector field, and  generates therefore a one-parameter group of isometries.
    However, we will see here, strikingly, that Brinkmann manifolds exhibit a Riemannian-like behavior: they are complete, and their isometry group tends to be compact (more precisely, the flow of $V$ tends to be relatively compact in the isometry group.).\medskip
	
	Our work generalizes T. Leistner and D. Schliebner's    completeness result of a special class of Brinkmann spacetimes, those with abelian holonomy \cite{LS}, also called pp-waves. From the dynamical point of view, we are studying 
	here  the opposite situation to the one 
	recently  considered  by C. Frances  who got  a quite precise description of Lorentzian manifolds 
	having an isometry group of exponential growth \cite{Fra}.

	\subsection{Completeness}
	Generic Lorentzian metrics on compact manifolds  are thought of to be  incomplete.   The historical example is the Clifton-Pohl torus $\R^2 - \{(0, 0)\} /_{ (x, y) \sim 2 (x, y)}$ endowed with the metric $\frac{dx dy}{x^2 + y^2}$ (see \cite{CR, RS-C1, RS-C2, S, Li} for various results on completeness of Lorentzian  surfaces).
	
	\subsubsection{Completeness results} In the sequel, we will mention 
	known completeness results (essentially all, to our best knowledge).
	One observes here that all these results 
	assume  some symmetry or a high local symmetry hypothesis.  
	
	\subsubsection*{Homogeneous case}
	The ``oldest'' one is  perhaps Marsden's result  \cite{Mar} stating completeness of compact homogeneous pseudo-Riemannian manifolds. Let us observe here that, contrary to the Riemannian case, indefinite homogeneous pseudo-Riemannian (non-compact)  manifolds fail to be complete in general, unless in some classes, 
	for example symmetric spaces and left invariant metrics on $2$-step nilpotent groups, as this was proved by Guediri
	\cite{Gue}.  See also \cite{sanchez-zeghib-left}, a recent work studying Lie groups with all left invariant pseudo-Riemannian metrics complete.
	
	\begin{example}\label{Brinkman homogene incomplet}
		$U = \{(x,y) \in \R^2, y >0\}, g = 2dxdy.$
		The subgroup of $O(1,1)  \ltimes \R^2$ consisting of Lorentzian affine transformations of the form $(x,y) \mapsto (\alpha x + b, \alpha^{-1}y), \alpha >0, b \in \R$, acts transitively on $U$, which is then homogeneous. It is however incomplete since it is a proper subset of $(\R^2, 2dxdy)$. Observe that this has a  Brinkmann structure, with parallel null vector field $V = \partial_x$.   Its Lorentzian metric is homogeneous, but it is not homogeneous for the Brinkmann structure, i.e. the group of isometries does not preserve the parallel vector field. 	
	\end{example}
	
	\subsubsection*{Constant curvature case}
	The most striking result is   Carri\`ere's  Theorem  \cite{Car} on completeness of  compact flat Lorentzian manifolds, and its generalization to the constant (sectional) curvature case by Klingler \cite{Kli}. \\
	Among locally homogeneous spaces, those of constant curvature can be interpreted as having 
	a maximal local isometry group.   
	There are   however  examples of incomplete compact locally homogeneous Lorentzian manifolds, e.g. a quotient  
	$\SL(2, \R) / \Gamma$, where $\Gamma $ is a co-compact lattice, and $\SL(2, \R)$ is endowed with a generic left invariant Lorentzian metric \cite{BM}.

	\subsubsection*{Timelike Killing field} Other results in the compact case assume existence of a few symmetry, e.g. a Killing field, say $V$,  
	but with a given constant causal character. For example, the Euler vector field whose flow is the homothetic action on $\R^2$,  determines a Killing field on the Clifton-Pohl torus, but with a varying  causal character.  
	Romero and Sanchez \cite{RS} proved completeness when $V$ is (everywhere) timelike i.e. of negative square length, actually merely assuming $V$ being a conformal Killing field.

	 \subsubsection*{Lightlike Killing field.} Taking a product of an incomplete Lorentzian manifold with the circle, equipped with a Riemannian metric, shows that the existence of a spacelike Killing field does not imply completeness.  But what if the Killing field is everywhere lightlike? There is a non-compact  $3$-dimensional Lorentzian homogeneous space, called  Lorentz-SOL geometry in \cite{DZ}, which is incomplete although having a lightlike Killing field. Even in the compact case, the presence of a lightlike Killing field does not ensure completeness, as shown in   \cite[Example 6.5]{Malek}, where a $4$-dimensional incomplete locally homogeneous Lorentzian manifold with a lightlike Killing field is constructed. In fact, compact incomplete examples exist starting from dimension $3$; we refer to \cite[Paragraph 7.3]{Malek} for such examples.

	\subsubsection*{Parallel fields} Now, what about the completeness question under the existence of a parallel vector field?  Recall that a vector field 
	$ V$ is parallel if its covariant derivative vanishes, i.e. $\nabla_X V=  0$, for any vector field $X$, where $\nabla$ is the Levi-Civita connection of the Lorentzian metric $g$.
	Equivalently, $V$ is invariant under parallel transport: for a smooth curve $c: [0, 1]\to M$, $\tau_c (V(c(0)) =
	V(c(1))$, where $\tau_c$ denotes the parallel transport $T_{c(0)} M \to T_{c(1)} M$. 
	Again, the question makes sense 
	only  in the lightlike case. This is  exactly what our main result answers, since  by definition, 
	Brinkmann spacetimes  are those admitting  a lightlike parallel vector field.

\subsubsection{Case of Brinkmann spacetimes} Before we delve into the main result on completeness, let us introduce some terminology:
\begin{definition}\label{Definition: Brinkmann spacetime}
(1) A Brinkmann spacetime  is a Lorentzian manifold $(M,g)$ admitting a lightlike parallel vector field $V$. We denote the structure by $(M,g,V)$. \\
(2) Let $(M,g,V)$ be a Brinkmann spacetime. A Brinkmann isometry is an isometry of $(M,g)$ preserving $V$. We denote by $\Iso(M,g,V)$ the group of Brinkmann isometries.   
\end{definition}
\begin{definition}
(1) We call Brinkmann-homogeneous a Brinkmann spacetime which is homogeneous in the sense of the Brinkmann structure, i.e. the group of Brinkmann isometries acts transitively. \\
(2) A compactly Brinkmann-homogeneous spacetime is a Brinkmann spacetime such that there exists a compact subset whose iterates by the Brinkmann isometry group covers all the space. 
\end{definition}
Now, we can state our result:
\begin{theorem}[Theorem \ref{theorem-main: compact Brinkmann is complete}]\label{completeness}
A compactly  Brinkmann-homogeneous spacetime is complete. In particular, both compact Brinkmann spacetimes and    Brinkmann-homogeneous spacetimes are complete. 		
\end{theorem}

Observe that Example \ref{Brinkman homogene incomplet} is homogeneous, but not Brinkmann-homogeneous (nor compactly Brinkmann-homogeneous). So if we assume homogeneity for the metric only, this example shows that we can have incompleteness.
\medskip
	
\noindent Some comments are in order:
	
\subsubsection*{pp-waves} Let $(M, g, V)$ be a Brinkmann Lorentzian manifold. Since $V$ is parallel, its  orthogonal distribution $V^\perp$ is invariant by the Levi-Civita connection and hence is integrable and has totally geodesic (and degenerate) leaves.  We will always note the so defined foliation $\F$. Each leaf of $\F$ inherits an induced  connection. Then,   pp-waves  are defined by the fact that all the $\F$-leaves are flat. These spacetimes  are very important to General Relativity. As previously mentioned, completeness of compact pp-waves was proved by   Leistner and Schliebner in \cite{LS}. It is this result that motivates our present work.  \\
Observe  that Brinkmann class is more flexible than the pp-waves one. For instance, the product of a Brinkmann spacetime with a Riemannian spacetime is still Brinkmann. One needs taking the product with a flat Riemannian manifold in the pp-wave case.  We will also see in Par. \ref{Subsection: M*} a construction of a Brinkmann structure on $M^*$, the bundle of orthonormal frames of $V^\perp$. The so-constructed structure is never a pp-wave (say for $\dim M > 4$).

\subsubsection*{Ehlers-Kundt conjecture}  It is  somewhat unrelated to our subject since it concerns non-compact spaces, still it is a completeness question on Brinkmann spaces,  which aims to characterize complete Ricci-flat (global) pp-waves. A pp-wave metric can be written locally in the following special form on $\R^2 \times \R^n$: 
$g = 2 du (dv + H(u, z) du) + \mathsf{euc}_{\R^n}$, where $(u, v)  \in \R^2$, $z \in \R^n$ and 
$\mathsf{euc}_{\R^n}$ denotes the Euclidean metric $\Sigma (dz^i )^2$. The parallel vector field is $V = \frac{\partial}{\partial v}$ (this derives from the fact that $H$ does not depend on $v$). 
The Ehlers-Kundt conjecture states that in the global case, with 
$n=2$, i.e., when $H$ is defined globally on $\R \times \R^2$
(also referred to as being in `standard form'), if 
$g$ is Ricci-flat, then $g$ is complete if and only if $H$ is quadratic in $z$, in which case $g$ is called a plane wave (see \cite{FS} for the most recent results on the subject).
Observe that Ricci-flatness is equivalent to harmonicity of $H$ with respect to $z \in \R^2$. This fact together with  the reduction of the geodesic equation to a mechanical system, that will be discussed below, show the post-Newtonian character of pp-waves, which applies also   to Brinkmann spaces. The latter spaces also admit special but more complicated local charts for their metrics, see Par. \ref{Local_coordinates}. Returning to the Ehlers-Kundt conjecture, the compact case, i.e., when the pp-wave in standard form  $(\R^2 \times \R^n, g)$ admits a compact quotient which is also a pp-wave, has been proven in \cite[Corollary 1]{LS}. The general non-compact case remains open (see \cite{FS}).

	\subsubsection*{Weakly  Brinkmann} A weakening of the Brinkmann  condition is to assume existence of 
	a parallel direction field $E$, i.e. existence of a 1-dimensional sub-bundle $E \subset TM$ invariant under 
	parallel transport. In a similar manner, one can define a weak version  of pp-waves. An example given in \cite{LS} shows that the latter can be incomplete, even when compact.   It is also asked in this paper if the leaves of the  lightlike geodesic foliation 
	$\F$ (tangent to  $E^\perp$) are complete in the compact case.  In a recent paper \cite{Malek}, it is shown that when the manifold is compact, the $\F$-leaves are complete as soon as the ($1$-dimensional) leaves tangent to $E$ are complete. Moreover, an incomplete compact example is constructed.  \;  
	More abstractly, the Brinkmann class, which naturally generalizes  pp-waves, is in  turn part of the broader class of (locally) Kundt spacetimes, defined as those having a codimension one lightlike geodesic foliation $\F$ 
	(see for instance  \cite{BMZ}). It is shown in \cite{Malek} that the aforementioned result holds in the more general setting of (compact) locally Kundt spacetimes.

	\subsubsection*{An associated mechanical system} Other works that should be quoted here are 
	\cite{CFS} and \cite{CRS}.  They relate completeness of pp-waves to that of a 
	second order differential equation on a   Riemannian manifold $(M_0,g)$, defined by a general class of forces:  
	\begin{align}\label{Equation E_0}
		\nabla_{\dot{\gamma}(t)} \dot{\gamma}(t) = F_{(\gamma(t),t)} \dot{\gamma}(t) + X_{(\gamma(t),t)},
	\end{align}
	where $\nabla$ denotes the Levi-Civita connection of $g$, $F$ a $(1,1)$-smooth tensor field  and $X$ a smooth vector field on $M_0 \times \R$. 
	Observe that this is a linear perturbation of the geodesic equation of $(M_0, g)$. 
	Sufficient conditions on $F$ and $X$ for the trajectories to be complete are given in 
	\cite{CRS}, assuming $(M_0, g)$ complete, and also in \cite{Impulsive}, both in the smooth case and in low regularity cases (assuming the metric functions of distributional nature).
	
	In the general Brinkmann case, one meets, in adapted local coordinates,  equations of the form:
	\begin{align}\label{Equation E_1}
		\nabla_{\dot{\gamma}(t)}^t \dot{\gamma}(t) = F_{(\gamma(t),t)} \dot{\gamma}(t) + X_{(\gamma(t),t)}.
	\end{align} 
	Now, $\nabla^t$ is the Levi-Civita connection of a (locally defined) Riemannian metric $g_t$, varying with time. 
	So our situation is noticeably different from the one investigated in \cite{CRS}, first because there is no globally defined Riemannian metric on $M$, secondly because of the time dependency of the connection involved
	in Equation (\ref{Equation E_1}).

	\vspace{0.5cm}
	In addition to the previous comments, let us mention here that in dimension $3$, all Brinkmann manifolds are pp-waves. Their completeness in the  ``roughly homogeneous'' case, i.e. when the isometry group acts transitively, without preserving the parallel vector field, is investigated in the literature, for instance in \cite{Rec_curvature}, where they are referred to as Lorentzian ``Walker spaces''. In particular, there are maximal incomplete homogeneous examples, meaning that they admit no embedding into a larger space. These spaces are (incomplete) homogeneous plane waves. Let us mention that what some references call Walker spaces can be seen as generalizations of Brinkmann spacetimes in a general pseudo-Riemannian context. More recently, a $C^2$-maximality result has been established in any dimension for simply connected, non-flat homogeneous plane waves (all of which are incomplete in the roughly homogeneous case).

	\subsection{Reduction to an (almost) locally homogeneous situation}

	Beyond completeness, we want to understand how Brinkmann spacetimes are made up.
	In Section \ref{Cartan_Connection}, inspired by \cite{Fra}, we will prove the following result
	which will be next used to study the dynamics of $V$.
	
\begin{theorem}\label{reduction}
Let $(M, g, V)$ be a compact  Brinkmann spacetime. Then  
$M$ admits a  core $N$, a closed submanifold of $M$, invariant under a finite index subgroup of $\Iso(M, g, V)$.  

\noindent Description of the core $N$: there exists an open subset $\tilde{U}$ of $\tilde{M}$, invariant under all local isometries of $(\tilde{M}, \tilde{g}, \tilde{V})$, containing a lift $\tilde{N}$ in the universal cover $\tilde{M}$, and such that  
		
\begin{enumerate}
  \item  Either $N$ is a  locally homogeneous  Lorentzian Brinkmann closed submanifold of $M$; more 
  precisely,  the lift $\tilde{N}$ is the orbit of a finite index subgroup of $G = \Iso^0(\tilde{U}, \tilde{g}, \tilde{V})$, which has finite index in $\Iso(M, g, V)$. 
			
  \item Or $N$ is    a  locally co-homogeneity  one  Lorentzian Brinkmann  closed submanifold of $M$,  with boundary, which is a trivial fibration over an interval. The fibers are lightlike  totally geodesic and locally homogeneous. More precisely,  in the universal cover $\tilde{M}$,  the lift of the fibers of $\tilde{N}$ are the orbits of a finite index subgroup   of $G = \Iso^0(\tilde{U}, \tilde{g}, \tilde{V})$, which has finite index in $\Iso(M, g, V)$.
\end{enumerate} 
		
In each case, the core $N$ or its (codimension 1) fibers have the form $\Gamma \setminus G / I$, where $I$ is  a closed subgroup in $G$, and $\Gamma$ is a discrete subgroup of $G$ acting properly and freely on $G/I$. 
		
Actually, the core $N$ is not unique: in each of the previous cases, there exists an open subset of $M$ that is trivially fibered over an open subset of $\R^k$, for some $k \in \N$, with fibers that are submanifolds like $N$. 
\end{theorem}

\subsection{Dynamics}
Besides completeness issues, the other  interesting topic on Lorentzian manifolds is to understand when their isometry  group is non-compact (assuming the  manifold compact), see for instance \cite{Zim, Gro, AS1, AS2, Ze1, Ze2, Ze3, PZ, Fra1}. 
The   more recent article \cite{Fra} studies actions of discrete groups with exponential growth, a situation somehow opposite to ours here. Indeed, we want  to know here whether $V$ is ``essential '' or not. In other words, we ask the question if the flow of $V$ is equicontinuous (or not), or equivalently, if it preserves (or not) an auxiliary Riemannian metric.   If it does not preserve such a Riemannian metric, then it is really of  Lorentzian  nature.   Actually, one may  ask such a question for a parallel vector field, say $W$, not necessarily null as in the Brinkmann case: 
\begin{itemize}
\item[-] In the case where $W$ is timelike, it is obviously equicontinuous: just  apply a ``wick rotation'' to the Lorentzian metric $g$ to make it Riemannian; precisely, keep $g$ on $W^\perp$ and multiply it by $-1$ on $\R W$.
\item[-] There are interesting non-equicontinuous  examples in the case where $W$ is spacelike.  For instance, a hyperbolic matrix $A \in \SL(2, \Z)$ preserves a flat Lorentzian metric on $\mathbb T^2$. Its suspension is a parallel spacelike flow on a flat Lorentzian $3$-manifold $\mathbb T^2_A$. It is Anosov and so far away from being equicontinuous.
\item[-] For   lightlike parallel vector fields, we will prove, partially, that their flows  are equicontinuous. 
For this,  we will refer to the second situation in    Theorem \ref{reduction} as the degenerate case. Also, in order to state our next result, notice that the foliation $\F$ determined by $V^\perp$ is defined by a non-singular closed 1-form. Hence, it is either minimal, i.e. all its leaves are dense, or all its leaves are closed.
\end{itemize}

\begin{theorem}\label{dynamics}
Let $(M, g, V)$ be a compact  Brinkmann spacetime. Then, the flow of   $V$ is equicontinuous, that is its closure in $\Iso(M, g)$ is compact,  in each of the two following cases:

		-   the foliation determined by $V^\perp$ is not minimal;

		-    the degenerate case of the   reduction Theorem \ref{reduction}, i.e. when the core $N$ is not locally homogeneous, but rather has
		local co-homogeneity one. 		
\end{theorem}

\subsubsection{Dynamics in the locally homogeneous case} We think that this equicontinuity result extends to the general case of compact Brinkmann spacetimes.  Our results reduce the proof of equicontinuity to the locally homogeneous case, that is $M$ has the form $\Gamma \setminus G / I$ where:
\begin{itemize}
    \item $I$ is  a closed subgroup in $G$ (in fact contained in the nil-radical of $G$). 
    \item The $G$-action on $G/I$ preserves a Lorentzian metric $\tilde{g}$. 
    \item $\Gamma$ is a discrete sub-group of $G$ acting
	properly  freely and co-compactly  on $G/I$. 
    \item $Z$  is a central $1$-parameter subgroup of $G$ defining a parallel vector field $\tilde{V}$ on $G/I$.
\end{itemize}	
	
The question, which we guess has a positive answer,  is whether the $Z$-action on $M$ lies in a compact torus 	${\bf T}\subset \Iso(M, g)$?

	This would allow one to give a somewhat exact description of $M$. One starts with a compact manifold 
	$M$ with a toral action, and sees if a $1$-parameter group $Z \subset {\bf T}$ can be parallel
	for some Lorentzian metric. 
	
	It could appear paradoxical  to  be able to  handle  equicontinuity in the degenerate case, but not in the 
	locally homogeneous case! The reason is that this sympathetic algebraic form  $\Gamma \setminus G / I$ 
	hides arithmetic and dynamical formidable difficulties due to the discreteness of  $\Gamma$ and the non-compactness of $I$.   
	
	\subsubsection{Cahen-Wallach spaces}
	They are (indecomposable)  Brinkmann (globally) symmetric 
	spaces. More precisely,   they are global metrics  
	on $\R^2 \times \R^n$ given by a formula:
	$g = 2 du (dv + H(z) du) + \mathsf{euc}_{\R^n}$, where $(u, v)  \in \R^2$, $z \in \R^n$,  $\mathsf{euc}_{\R^n}$ is the Euclidean metric $\Sigma (dz^i )^2$, and $H$ is a quadratic form on $z$.
	Their discrete groups $\Gamma$  giving rise to compact quotients are investigated 
	by I. Kath and M. Olbrich in \cite{KL}. It is proved in their cases, but after a long algebraic 
	preparation, that their parallel vector field is in fact periodic, that is it corresponds to a 
	$\mathbb S^1$-action.

	\subsubsection{Flat case} Let us first note that if a Lorentzian manifold $(M, g)$  has  $d >1$,   linearly independent null parallel vector fields, then its universal cover splits as $\tilde{M}= \tilde{N} \times \R^d$, where $\tilde{N}$ is Riemannian, 
	$\R^d$ is flat Lorentzian, and $\pi_1(M)$ acts by translation along  the factor $\R^d$. One can then find a null parallel vector field 
	whose closure in $\Iso(M, g))$  is a torus of dimension $d$, or maybe more.
	
	Consider now the case where $M$ is flat with exactly one (up to scaling) null parallel vector field. Then, by  Carri\`ere completeness  Theorem,  $M $ is a quotient of the Minkowski space $\Mink^{1, n}$  by a discrete group $\Gamma$.
	Many results are known regarding $\Gamma$. In particular,  $\Gamma$  is solvable, in fact it has  a ``crystallographic hull'' $L$  containing it, 
	a solvable connected Lie  $L$, 
	acting simply transitively and isometrically on $\Mink^{1, n}$ \cite{FGH, GKa, GM}. One can deduce from this that if $M$ has exactly one null parallel field, then as in the Cahen-Wallach case, the so defined flow is periodic.

\subsubsection{A non-periodic example}
	
Consider on  $\tilde{M}=  (I \times \R) \times \R^n$, parameterized by $(u, v, z^1, \ldots, z^n)$, a metric of the form:  $g = 2 du dv  +   \alpha_{ij}(u)dz^i dz^j$ where $\alpha: u  \in I \to  ( \alpha_{ij})_{i,j}(u) $ is a curve of positive definite matrices,  and  $I$ is either an interval or the circle $\mathbb S^1$.  	These are pp-waves, and non-flat for a generic   $\alpha$. They are even plane waves (written in Rosen coordinates), see \cite[p. 6]{Rosen}. The abelian group $\R^{n+1} $ acts  isometrically, trivially on $u$ and  by translation on the coordinates $(v, z^1, \ldots z^n)$. This action is transitive on the $u$-levels.  Let $\Lambda$ be a lattice in $\R^{n+1}$ with respect to which the $v$-axis is irrational, that is the translation along $v$ determines a minimal linear flow on $\R^{n+1} / \Lambda$.  The quotient $\tilde{M} / \Lambda$  has the topology of $I \times \mathbb T^{n+1}$, on which the flow of $\frac{\partial}{\partial v}$ acts minimally on the $\mathbb T^{n+1}$-factor.  It will be noted 	$M(\alpha, \Lambda)$. In the case $I = \mathbb S^1$, one gets compact pp-waves with topology $\mathbb T^{n+2}$, for which   closures of the parallel flow are given by a $\mathbb T^{n+1}$-action. 
 
In fact, it turns out that any Brinkmann spacetime as in    the situation of Theorem \ref{dynamics}  is built up by pieces of the form $M(\alpha, \Lambda)$. 	(This may be extracted from our  last  Sections  \ref{Further results}, \ref{closure foliation} and 
\ref{end of proof}. Details will be published elsewhere).

\subsection{Organization of the article}
	
${}$
	
Section \ref{Subsection: generalities on Brinkmann spacetimes} contains some generalities about Brinkmann spacetimes. In Sections \ref{Subsection: M*} and \ref{Subsection: Existence of totally geodesic surfaces in a Brinkmann spacetime}, we give some properties on the local geometry of Brinkmann spacetimes; these properties are interesting on their own, but they will also be needed in the development of the rest of the article.  In particular, we introduce in Section \ref{Subsection: M*} some principal bundle over a Brinkmann spacetime $(M,V)$ with the property that the induced action of the subgroup of $\Iso(M)$ preserving $V$ has unipotent isotropy; this feature will serve in the study of the $V$-dynamics on $M$, carried out in Sections \ref{Further results} and \ref{closure foliation}. Among submanifolds of a pseudo-Riemannian manifold, totally geodesic ones are fundamental. In Section \ref{Subsection: Existence of totally geodesic surfaces in a Brinkmann spacetime}, we prove the existence of many such  submanifolds in a Brinkmann spacetime. 
 
We will then use  all this in Section \ref{Section: Completeness along the leaves} to give a synthetic proof of completeness through the lightlike geodesic foliation $F$ orthogonal to $V$. 

Section \ref{Section: Geodesic equation} gives the geodesic equations on a Brinkmann spacetime,  and Section \ref{Section: Analysis of geodesic equation} focuses on the proof of Theorem \ref{completeness} on the completeness of compactly  Brinkmann-homogeneous spacetimes, through an analysis of the geodesic equation. 

In Section \ref{Cartan_Connection}, we prove Theorem \ref{reduction} on the existence of a core $N$, and   Section \ref{Further results}  gives more details on the structure of this core. 

Sections  \ref{closure foliation}  and  \ref{end of proof}   contain the proof Theorem \ref{dynamics}. Section \ref{closure foliation} seems to be the most technical. One deals with a  homogeneous lightlike space $G/I$, and a discrete subgroup  $\Gamma \subset G$  acting properly co-compactly. 
A proven approach to such a problem is to find a kind of connected envelope $H$  containing $\Gamma$ and still acting properly. The difficulty in implementing this idea lies  on one hand  in the possible existence of a semi-simple compact factor of $G$, and on the other hand the fact that the   radical of $G$ is not nilpotent.  What helps us  here is that $V$ defines a $1$-dimensional  transversally Riemannian foliation, say  $\mathcal V$.  As this is the case of any transversally Riemannian foliation, the closure of the $\mathcal V$- leaves defines a  transversally Riemannian foliation (possibly singular) $\overline{\mathcal V}$.    By a result of Y. Carri\`ere \cite{Car1, Car2},  because $\dim \mathcal V = 1$,    $\overline{\mathcal V}$ has toral leaves. The algebraic richness of our situation allows us to find such a syndetic hull $H$.

\subsection*{Acknowledgement}
We thank the three anonymous referees for their thorough evaluation and valuable comments and suggestions, which helped improve the paper. Special thanks to one of the referees for suggesting the use of Observation \ref{observation: when Killing field is parallel} to prove Fact \ref{V* is parallel}, providing an alternative and easier proof.

\section{ Brinkmann Geometry}\label{Section: Brinkmann geometry}
\subsection{Preliminaries and local coordinates}\label{Subsection: generalities on Brinkmann spacetimes}
In this section, we derive some interesting properties on the local geometry of Brinkmann spacetimes.
	
Throughout this section, $(M,g,V)$ is a Brinkmann manifold of dimension $n+1$. We assume $V$ to be complete (which is only needed in Par. \ref{Subsection: Existence of totally geodesic surfaces in a Brinkmann spacetime}).  Denote by $\mathcal{V}$ the $1$-dimensional foliation defined by $V$, and by $\F$ the foliation of codimension $1$ defined by the parallel distribution $V^{\perp}$. 
	
\subsubsection{Transverse Riemannian structure} 
The leaves of $\F$ are lightlike submanifolds of $M$ foliated by the $1$-dimensional foliation $\mathcal{V}$. Since $V$ is parallel, $\F$ is a geodesic foliation. Therefore, the foliation $\mathcal{V}$  along any leaf $F$ of $\F$ admits a transverse Riemannian structure invariant by the local flow of any vector field tangent to $V$.  Locally, it is given by any $(n-1)$-submanifold $S$ contained in $F$ and transversal to $\mathcal{V}$, endowed with the Riemannian metric induced by $g$.

\subsubsection{Local coordinates} \label{Local_coordinates}
Brinkmann spacetimes admit two different adapted local coordinates, known in the literature as Brinkmann coordinates (see for instance the construction in \cite[Par. 4]{San}), and Rosen coordinates. In what follows we recall the construction of the Rosen coordinates. 

\begin{fact}\label{Fact: Rosen coordinates}
A Lorentzian spacetime $(M,g)$ is a Brinkmann spacetime of dimension $n+1$ if and only if there is a globally defined vector field $V$ on $M$, such that any point $p \in M$ admits a coordinate chart $(u,v,x^1,\dots,x^{n-1})$  in which the metric takes the form 
$$g=2 \dd u \dd v + g_{ij}(x,u) \dd x^i \dd x^j,$$
with $V= \partial_v$. 	These coordinates are refered to in \cite{Rosen} as Rosen coordinates.
\end{fact}

\begin{proof}
Rosen coordinates are defined as follows. Consider $F_0$ a leaf of $\F$, and take any $(n-1)$-submanifold $\Omega_0$ in $F_0$ transversal to $V$. The induced metric on $\Omega_0$ is a Riemannian metric. Denote by $Z$ the null vector field along $\Omega_0$ which is orthogonal to $\Omega_0$ and transversal to $V$, such that $g(Z,V)=1$. Such a vector field is uniquely defined. Indeed, for every $x \in \Omega_0$, $T_x \Omega_0$ is spacelike of dimension $n-1$, hence $T_x \Omega_0^{\perp}$ is Lorentzian of dimension $2$. Thus $T_x \Omega_0^{\perp}$ contains a (unique) second lightlike direction, other than that of $V$. So $Z$ is uniquely defined when adding the assumption that $g(Z,V)=1$. Now, denote by $\Omega$ the hypersurface transversal to $V$ such that $\Omega \cap F_0 = \Omega_0$, obtained by taking for every $x \in \Omega_0$ the (null) geodesic with initial velocity $Z_x$. Choosing $\Omega_0$ smaller, if necessary, these geodesics define a local flow on $\Omega$; we denote again by $Z$ its infinitesimal generator. Denote by $\phi^v$ (resp. $\psi^u$) the local flow of $V$ (resp. $Z$). 
		
For $q \in \Omega$, define $u(q)$ to be the unique real such that $ \psi^{-u(q)}(q) \in \Omega_0$. Take some local coordinates $(x^1,\ldots,x^{n-1})$ in $\Omega_0$ and extend them in a neighborhood of $\Omega_0$ of $\Omega$ by setting for $q \in \Omega, x(q)=x(q_0)$, where $q_0 = \psi^{-u(q)}(q)$. We obtain local coordinates $(u,x^1,\ldots,x^{n-1})$ on $\Omega$.  Finally, define a local chart in a neighborhood of $p$ by extending the latter coordinates by the flow of $V$. This gives a diffeomorphism $\tilde{f}=(u,v,x^1,\ldots,x^{n-1})$ from a neighborhood $U$ of $p$ into an open subset of $\R^{n+1}$, by setting for $q \in U$, $f(q)=(u(q), v(q), x^1(q),\ldots,x^{n-1}(q))$, where $\phi_{-v(q)}(q) \in \Omega$, $\psi_{-u(q)}(q) \in  F_0$ and $x(q) = x(\phi_{-v(q)} \circ \psi_{-u(q)}(q))$.  
		
Again by taking $\Omega_0$ smaller, we can assume that the flows of $V$ and $Z$ are defined for $\vert v \vert, \vert u \vert < \epsilon$, for some $\epsilon >0$. If we take $\Omega_0$ to be a metric ball of radius $r$ (with respect to the induced Riemannian metric), this defines a differomorphism of a neighborhood of $p$ onto a set $B_{n-1}(0,r) \times I \times J$, where $B_{n-1}(0,r)$ is the open ball of center $0$ and radius $r$ in $\R^{n-1}$, and $I$ and $J$ are open intervals of $\R$. We claim that the metric in these coordinates has the given form. First, since $V$ acts isometrically, the orbits of $\partial_u$ are null geodesics, hence $g(\partial_v, \partial_u)$ is constant (Clairaut's constant) equal to $1$, and $g(\partial_u, \partial_u)=0$. Next, the local flow of $\partial_u$ leaves invariant the distribution tangent to $\mathcal{F}$. This is a consequence of a general fact: if $\alpha$ is a closed $1$-form such that $\alpha(Z)=1$, then the flow of $Z$ leaves invariant the distribution $\ker \alpha$. Here, take $\alpha := g(V,.)$ which is closed since $V$ is parallel. Consequently, $\partial_{x_i}, i=1,\ldots,n-1,$ are everywhere tangent to  $\mathcal{F}$, hence  $g(\partial_v, \partial_{x_i})=0$. We are left with the proof that $g(\partial_u, \partial_{x_i})=0$. We have $\partial_u \cdot g(\partial_u, \partial_{x_i})= g(\nabla_{\partial_u} \partial_u, \partial_{x_i})+g( \partial_u, \nabla_{\partial_u} \partial_{x_i})$. The first term vanishes since 
the orbits of $\partial_u$ are geodesics. And 
$2 g( \partial_u, \nabla_{\partial_u} \partial_{x_i})=2 g( \partial_u, \nabla_{\partial_{x_i}} \partial_{u})= \partial_{x_i} \cdot g(\partial_u, \partial_u)=0$.
		
Conversely, one can check by a computation of Christoffel symbols (or using Observation \ref{observation: when Killing field is parallel}) that in the local coordinates, $\partial_v$ is lightlike and parallel for a metric of the announced form.
\end{proof}

\subsection{The $\O(n-1)$-principal bundle $M^*$ over $M$.}\label{Subsection: M*}  

Let $(M,g,V)$ be a Brinkmann spacetime.  The goal of this section is to associate to $M$ a natural new Brinkmann spacetime $(M^*,g^*,V^*)$ with an isometric submersion $M^* \to M$, such that:  
\begin{itemize}
    \item[$\bullet$] The new foliation $\mathcal{V}^*$ is transversally parallelizable on each leaf of $\F^*$. This follows the standard construction in the theory of transversally Riemannian foliations. Here, we provide a variant of the construction. The parallelizability will be used in Section \ref{Section 8} (see Par. \ref{Par. 8.1}). 
    \item[$\bullet$] Any isometry of $M^*$ preserving $V^*$ and commuting with the submersion acts trivially on $(V^*)^\perp / \R V^*$. This will be used in an essential way in Section \ref{Further results}. 
\end{itemize} 
\bigskip

Let $(M,g,V)$ be a Brinkmann spacetime of dimension $n+1$. Define the vector bundle $$E := V^{\perp}/ \R V \to M,$$ which is equipped with a positive definite metric induced by $g$, $$g_E([X], [Y]) = g(X,Y).$$ A linear orthonormal frame $r_p$ of $E$ at a point $p \in M$ is an ordered orthonormal basis	$r_p= (e_1 + V,\dots,e_{n-1} + V)$ of the vector space $E_p$ for the Riemannian metric $g_E$. So $(e_1,\dots,e_{n-1})$ is a $g$-orthonormal family of vectors contained in $V^{\perp}$. 
 
Let $M^*$ be the set of all orthonormal frames of $E$ at all points of $M$, and denote by $$\pi: M^* \to M$$ the natural projection which maps $r_p$ to $p$.  This is a $\O(n-1)$-principal bundle. For each $r \in  M^*$, let $G_r$ be the subspace of $T_r M^*$ consisting of vectors tangent to the fiber through $r$. The Levi-Civita connection associated to $g$ gives a connection on $M^*$	i.e. a horizontal $\O(n-1)$-invariant distribution $\H$ such that for every $r \in M^*$, $T_r M^* = \H_r \oplus G_r$ (direct sum). Denote by $\omega$ the connection form: it is a $\mathfrak{o}(n-1)$-valued $1$-form on $M^*$, and by $\sigma: \mathfrak{o}(n-1) \to \chi(M^*)$ the Lie algebra homomorphism that maps $\mathfrak{a} \in \mathfrak{o}(n-1)$ to the fundamental vector field on $M^*$ associated to $\mathfrak{a}$ .  
\medskip

Observe that in dimension $3$, $M^*=M$.
\subsubsection{Preliminary facts}
Let $\phi^t$ be the $1$-parameter group generated by $V$. The flow of $V$ induces an action on $L(M)$, the bundle of orthonormal frames at points of $M$. Since the flow of $V$ acts isometrically and preserves $V$, this flow preserves $V^\perp$ and induces an action on the linear frames of $E$.  Therefore, for each $t$, $\phi^t$ induces a	transformation $\psi^t$ of $M^*$.	Thus we obtain a global $1$-parameter group of transformations $\psi^t$  of $M^*$, which induces a vector field on $M^*$ that we will denote by $V^*$.
	
\begin{fact}
If $V$ is parallel then $V^*$ is horizontal, i.e. it is the horizontal lift of $V$.
\end{fact}
\begin{proof}
Fix an orthonormal frame $r_p$ of $E_p$ at a point $p \in M$, and consider the orbit $\phi_p(t)$ of $V$ such that $\phi_p(0)=p$.  Write $r_p = (X_1(p) + V,\ldots,X_{n-1}(p)+V)$. For every $i=1,\ldots,n-1$, consider a curve tangent to $\mathcal{F}$ whose tangent vector at $p$ is $X_i(p)$, and denote again by $X_i$ its tangent vector field.	Consider an extension of $X_i$ in a neighborhood of the curve, which is tangent to $\F$ and invariant under the flow of $V$. We denote this extension by $\widehat{X_i}$. Since $V$ is parallel, the restriction of $\widehat{X_i}$ to the integral curve $\phi_p(t)$ of $V$ through $p$ is the parallel transport of $X_i(p)$ along $\phi_p(t)$. It follows that the curve $c(t)= (\phi_{p}^t)_*(r_p)$ in $M^*$ is exactly the horizontal lift of $\phi_p(t)$ such that $c(0)=r_p$, and this proves the fact.  
\end{proof}

\begin{fact}\label{Fact: induced f* on M*}
Let $f \in \Iso(M,g,V)$. Then $f$ induces a bundle automorphism $f^*$ of $M^*$ (which commutes with the action of $\O(n-1$)), called the lift of $f$ to $M^*$. Moreover, $f^*$ preserves the fundamental vector fields on $M^*$, preserves $V^*$, and preserves the horizontal distribution $\H$.    
\end{fact}
\begin{proof}
Since $f$ preserves $V$, clearly $f$ induces an action on $M^*$.  This (left) action commutes with the (right) action of $\O(n-1)$ on $M^*$.
In particular, $df^*$ preserves the fundamental vector fields on $M^*$.
Indeed,  $f^*$ commutes with $R_A$ for every $A \in \O(n-1)$. So for every $s \in \R$ and every $\mathfrak{a} \in \mathfrak{o}(n-1), f^* \circ R_{\exp(s \mathfrak{a})} = R_{\exp (s \mathfrak{a})} \circ f^*$. Taking the derivative with respect to $s$ at $s=0$ yields $df^*(\sigma(\mathfrak{a})(r))= \sigma(\mathfrak{a})(f^*(r))$, for every $\mathfrak{a} \in \mathfrak{o}(n-1)$. Hence the claim.
Next, the fact that $f^*$ preserves $V^*$ follows directly from the fact that $f$ preserves $V$. 
Now, let $X^* \in T_{r_0} M^*$ be a horizontal tangent vector at $r_0 \in M^*$. And let $r(s)$ be a horizontal curve in  $M^*$ tangent to $X^*$ at $r_0$. It is a horizontal lift of the curve $\gamma(s) = \pi(r(s))$ on $M$. Consider the curve $f^*(r(s))$. Since $r(s)$ is a parallel frame field along $\gamma(s)$, and $f$ is isometric,  the curve $f^*(r(s))$ is a parallel frame field along $f(\gamma(s))$.  It follows that $f^*(r(s))$ is a horizontal lift of $ f(\gamma(s))$.  Hence $df^*(X) = \frac{d}{ds} f^*(r(s))$ is a horizontal vector. 
\end{proof}

\subsubsection{A Brinkmann spacetime structure on $M^*$}
In this paragraph, we will prove the following proposition:
\begin{proposition}
There is a natural $\O(n-1)$-invariant Lorentzian metric $g^*$ on $M^*$ for which $(M^*,  g^*, V^*)$ is a Brinkmann spacetime.
\end{proposition}
We can define a $\O(n-1)$-invariant Lorentzian metric on $M^*$ in the following way: let $h_0$ be the positive definite inner product on $\mathfrak{o}(n-1)$ given by the Killing form. And set for $X^*, Y^* \in T_r M^*$:
$$g^*_r(X^*,Y^*) := g_{\pi(r)}( d_r \pi(X^*), d_r \pi(Y^*)) + h_0(\omega(X^*), \omega(Y^*)).$$
So 
	
$\bullet$ for every $r \in M^*$, $d_r \pi: \H_r \to T_r M$ is a linear isometry. 
	
$\bullet$ the vertical and the horizontal subspaces of $T_r M^*$ are orthogonal for every $r \in M^*$.

\begin{fact}\label{V* is a Killing vector field}
$V^*$ is a null Killing vector field for $(M^*,g^*)$ that preserves the distribution $\H$. 
\end{fact}
\begin{proof}
Here, we only need that $V$ is a Killing vector field. Applying Fact \ref{Fact: induced f* on M*} to the flow $\phi^t$  of $V$ shows that the flow $\psi^t_*$ of $V^*$ preserves any $\omega$-constant vertical vector field on $M^*$.  Next, let $X^* \in \Gamma(TM^*)$ be a horizontal vector field on $M^*$. Write $\phi_* \circ d\pi = d\pi \circ \psi_*$. Again, it follows from Fact \ref{Fact: induced f* on M*} that  the distribution $\H$ is invariant by the flow of $V^*$.
So $\psi^t_*(X^*)$ is a horizontal vector field that projects to $\phi_*^t \circ d\pi(X^*)$. Hence, $g^*(\psi^t_*(X^*),\psi^t_*(X^*)) = g(\phi^t_* \circ d\pi(X^*), \phi^t_* \circ d\pi(X^*))$. On the other hand, $g^*(X^*,X^*) = g(d\pi(X^*), d\pi(X^*))= g(\phi^t_* \circ d\pi(X^*), \phi^t_* \circ d\pi(X^*))$, since the flow $\phi^t$ of $V$ is isometric. This yields $g^*(X^*,X^*) =g(\psi^t_*(X^*),\psi^t_*(X^*))$ for every horizontal vector field $X^*$.  This proves that $V^*$ is Killing, and ends the proof.    
\end{proof}
	
\begin{fact}\label{V* is parallel}
$V^*$ is parallel for the Levi-Civita connection induced by $g^*$.
\end{fact}

To prove Fact \ref{V* is parallel}, we will use the following observation. 
\begin{observation}\label{observation: when Killing field is parallel}
A vector field $K$ is parallel if and only if it is Killing, and the $1$-form $\omega:=g(K,\cdot)$ is closed.  
\end{observation}

\begin{proof}
or any two vector fields $X$ and $Y$, we have $d\omega(X,Y)=X \cdot g(K, Y) - Y \cdot g(K, X) - g(K, [X, Y])$.  If $K$ 
 is a Killing vector field, this simplifies to $d\omega(X,Y)=2 g(\nabla_X K, Y)$, which implies that $K$ is parallel if and only if $\omega$ is closed.  
\end{proof}

\begin{proof}[Proof of Fact \ref{V* is parallel}]
Set $\omega :=g(V, \cdot)$ and $\omega^* := g^*(V^*,\cdot)$. 

Since the vector field $V^*$ is the horizontal lift of $V$, the pullback of $\omega$ by $\pi$ satisfies $\pi^*(\omega)=\omega^*$. As\, $\omega$ is a closed form, $\omega^*$ is also closed. Consequently, by Observation \ref{observation: when Killing field is parallel}, the vector field $V^*$ is parallel.
\end{proof}

\subsubsection{\textbf{Partial transversal parallelism}}
Define the new vector bundle $$E^*:= (V^*)^\perp/\R V^* \to M^*.$$ 
We will construct a transversal parallelism to $\mathcal{V}^*$, tangent to the foliation $\F^*$.  This is a parallelism on the vector bundle $E^*$. \\
Recall the decomposition $T_r M^* = \H_r \oplus G_r$ for every $r \in M^*$.  The bundle $E^*$ naturally inherits a decomposition into the sum of a horizontal and vertical sub-bundles. 
\medskip

\textbf{Notation:} Let $p: V^\perp \to E$ and $p^*: (V^*)^\perp \to E^*$ denote the canonical projections. Let $\overline{d\pi}: E^* \to E$ be the projection induced by $d \pi: TM^* \to TM$.
\medskip

Take $Y_1,\ldots,Y_{N} $ a basis of the Lie algebra $\mathfrak{o}(n-1)$, and consider $\sigma(Y_1),\ldots,\sigma(Y_{N})$ the associated vertical vector fields on $M^*$. They belong to $\Gamma(T \F^*)$ and induce a frame field $([\sigma(Y_1)],\ldots,[\sigma(Y_{N})])$ on the vertical sub-bundle of $E^*$. 
\medskip

Define a frame field $[X_1^*],\ldots,[X_{n-1}^*]$ on the horizontal sub-bundle of $E^*$ as follows. Let $r=(r_1,\ldots,r_{n-1}) \in M^*$. For each $i$, choose $X_i(r) \in V^\perp(\pi(r))$ such that $p(X_i(r))=r_i$. Let $X_i^*(r)$ be the unique horizontal vector at $r$ such that $d_{r} \pi(X_i^*(r)) = X_i(r)$. Observe that if two vectors in $V^\perp$  have the same $p$-projection to $E$, their horizontal lifts to $(V^*)^\perp$ have the same $p^*$-projection to $E^*$. So  $[X_i^*(r)]:= p^*(X_i^*(r)) \in E^*(r)$ is well defined (independent of the choice of $X_i(r)$).  It is the unique horizontal vector in $E^*(r)$ that projects to $r_i$ through $\overline{d\pi}$.
\medskip

The vector fields $([\sigma(Y_1)],\ldots,[\sigma(Y_{N})], [X_1^*],\ldots,[X_{n-1}^*])$ define a parallelism on $E^*$, such that $\overline{d\pi}([X_i^*](r))=r_i$.
\bigskip

Let $f \in \Iso(M,g,V)$. By Fact \ref{Fact: induced f* on M*}, the differential of $f^*$ preserves $V^*$, hence acts on $E^*$. We denote by $\overline{df^*}$ the induced map on $E^*$. Similarly, we use $\overline{df}$ for the map induced on $E$ by $df$. 
\begin{fact}\label{Fact: f* preserves parallelism}
The lift of a Brinkmann isometry of $(M,g,V)$ to $M^*$ induces a Brinkmann isometry of $(M^*,g^*,V^*)$ which preserves the parallelism on $E^*$. 
\end{fact}
\begin{proof}
Let $f \in \Iso(M,g,V)$. By Fact \ref{Fact: induced f* on M*}, $df^*$ preserves the vertical vector fields $\sigma(Y_i)$, which are $\omega$-constant. 
On the other hand, we have $\pi \circ f^*= f \circ \pi$. Moreover, $f^*$ preserves the horizontal distribution $\H$ by Fact \ref{Fact: induced f* on M*},  and $d\pi: \H \to TM$ is an isometry. Thus, the restriction of $df^*$ to $(\H, g^*_{\vert \H})$ is isometric.  Finally,  we obtain $f^* \in \Iso(M^*, g^*, V^*)$, by definition of $g^*$.  We will now prove that the induced map on $E^*$ preserves the parallelism on it. First, $\overline{df^*}$ preserves the vector fields $[\sigma(Y_i)]$. 
 Now, let $r:=(r_1,\ldots,r_{n-1}) \in M^*$. 
Set $Y_i^*:=\overline{df^*}([X_i^*])$, and $r':=f^*(r)$. 
We have $\overline{df}\,(\overline{d\pi}([X_i^*](r)) =\overline{df}(r_i)=r_i'$.   Writing $\pi \circ f^*= f \circ \pi$ yields $\overline{d \pi} \circ \overline{d f^*} ([X_i^*]) = \overline{d f} \circ \overline{d \pi}([X_i^*])$. Hence $\overline{d \pi} (Y_i^*(r')) = r_i'=\overline{d \pi}([X_i^*](r'))$. So $[X_i^*](r')$ and $Y_i^*(r')$ are horizontal vectors in $E^*(r')$ with the same $\overline{d \pi}$-projection to $E(p(r'))$, so they are equal. This yields $\overline{df^*}([X_i^*](r))=[X_i^*](f^*(r))$, which ends the proof. 
\end{proof}
\begin{corollary}\label{Corollary for unipotent isotropy}
Let $f \in \Iso(M,g,V)$. The differential of (the induced Brinkmann isometry) $f^* \in \Iso(M^*, g^*, V^*)$ at a fixed point $r \in M^*$ induces an element of the orthogonal group of $(T_r M^*, g^*_r)$ preserving $V^*(r)$ and acting trivially on $(V^*(r))^\perp/\R V^*(r)$. 
\end{corollary}
\begin{proof}
This is a straightforward consequence of Fact \ref{Fact: f* preserves parallelism}.  
\end{proof}

\subsection{Totally geodesic surfaces in a Brinkmann spacetime}\label{Subsection: Existence of totally geodesic surfaces in a Brinkmann spacetime}
	
Totally geodesic submanifolds appear only in special contexts. A generic pseudo Riemannian manifold does not admit any geodesic submanifold of dimension $>1$. In a Brinkmann spacetime, we will find many such  submanifolds. 
 
\subsubsection{Existence} Let $(M,g,V)$ be a Brinkmann spacetime. Consider the associated Grassmann $2$-plane bundle $Gr_2(M)$, obtained by replacing each fiber $T_xM$ by $Gr_2(T_xM)$,  the Grassmannian of the $2$-dimensional vector subspaces of 
$T_x M$. Denote by $\pi: Gr_2(M) \to M$ the natural projection. Define a sub-bundle $X$ of $Gr_2(M)$ in the following way: for each $x \in M$, the fiber at $x$ is
$$X_x := \{p_x \in Gr_2(M)_x, V \in p_x\}.$$ 
	
\begin{proposition}\label{Proposition: existence of totally geodesic surfaces in M}
For any $p_x \in X$, there is a totally geodesic flat surface through $x$ in $M$ whose tangent space at $x$ is $p_x$. 
\end{proposition}
\begin{proof}
Let $x \in M$ and $p \in X$ such that $\pi(p)=x$.  
Write $p =$ Vect$(Q,V)$. Let $\gamma$ be the geodesic in $M$ tangent to $Q$, and denote again by $Q$ its tangent vector field. Take the image of $\gamma$ by the flow of $V$. This defines in a neighborhood of $x$ a surface $S_p$ tangent to $p$ at $x$. Extend the vector field $Q$ by the flow of $V':=V_{\vert S_p}$. This defines a vector field (again denoted by) $Q$ tangent to $S_p$ which commutes with $V'$. Since $V$ is parallel, we have $\nabla_{V'} Q = \nabla_{Q} V' =0$,  where $\nabla$ is the Levi-Civita connection of $g$ (to define these properly, we consider an extension of $Q$ in a neighborhood of $S_p$, then $\nabla_{V} Q$ and $\nabla_{Q} V$ both vanish along $S_p$, independently of the extension). Furthermore, $\nabla_{Q} Q =0$, since $V$ acts isometrically on $M$, sending geodesics to geodesics. It follows that for any vector fields $X$ and $Y$ tangent to $S_p$, $\nabla_X Y$ is also tangent to $S_p$, which proves that $S_p$ is totally geodesic in $M$. Observe that $S_p$ is exactly $\exp_x(p)$ in a neighborhood of $x$, since radial geodesics tangent to $p_x$ are contained in $S_p$. Finally, it is a well known fact that a surface admitting a parallel vector field is actually flat. 
\end{proof}
	
\subsubsection{Flat bands in $(M,g,V)$}
Remember here that $V$ is assumed to be complete.
\begin{definition}[Flat bands] Let $S$ be a simply connected surface.  We say that $S$ is
\begin{enumerate}
    \item a degenerate flat band if $S$ is isometric to $(\R \times I, \dd y^2)$, $(x,y) \in \R \times I$, where $I$ is an open interval of $\R$. 
    \item a Lorentzian flat band if $S$ is isometric to $(\R \times I, 2\dd x \dd y)$, $(x,y) \in \R \times I$, where $I$ is an open interval of $\R$.
\end{enumerate}		
In this case, the vector field $V:=\partial_x$ is a null parallel vector field on $S$.
\end{definition}

\paragraph{\textbf{Length of a flat band}} Let  $(S,V)$ be a flat band, with $V$ its parallel null vector field. The foliation associated to $V$ admits a transverse Riemannian metric as a foliation of $S$, invariant by the action of $V$ (in fact by any vector field tangent to $V$ in the case of a degenerate band). Fix an orientation on $S$, together with a leaf $l_0$ of $V$.  
\medskip

1) Suppose $S$ is a degenerate flat band. Here, a geodesic transversal to $V$ is spacelike. We define the forward (resp. backward) length of the band to be the length of (any) maximal geodesic $\gamma: [0,T[ \to S$ transversal to $V$ such that $\gamma(0) \in l_0$, and $(V, \dot{\gamma})$ is positively (resp. negatively) oriented. We denote them by $l^{+}(S)$ and $l^{-}(S)$ respectively. 
\medskip

2) If $S$ is a Lorentzian flat band, $\omega=\langle V,. \rangle$ is a closed $1$-form that defines the foliation of $V$. Integrating $\omega$ along curves transversal to $V$ defines a transverse Riemannian structure. Let $\gamma$ be a maximal geodesic transversal to $V$ defined on $[0,T[$, such that $\gamma(0) \in l_0$, and set  $l(S):= \vert \int_{0}^{T} \langle V,\dot{\gamma} \rangle \dd t \vert$. We define the forward (resp. backward) length of the band to be $l(S)$ if $(V, \dot{\gamma})$ is positively (resp. negatively) oriented. We denote them again by $l^{+}(S)$ and $l^{-}(S)$ respectively.
\medskip
	
\noindent We say that $S$ is an infinite flat band if both $l^{+}(S)$ and $l^{-}(S)$ are infinite.
\begin{fact}\label{Fact 2.12: S are flat bands}
Let $S$ be a totally geodesic surface in $M$ saturated by $V$. Then the universal cover of $S$ is a flat band. 
\end{fact}
To prove this, we need the following lemma.
\begin{lemma}
Let $S$ be a simply connected surface, equipped with a flat and torsion free connection $\nabla$.  Suppose that $S$ admits a complete parallel vector field $V$. Then there is a diffeomorphism $\phi: S \to \R \times I$, $(x,y) \in \R \times I$ ($I$ is an interval of $\R$), where $\R \times I$ is equipped with its affine structure, and $\phi_* V = \partial_x$. 
\end{lemma}
\begin{proof}
The surface $S$ admits a $(\Aff(\R^2),\R^2)$-structure. We shall prove that the developing map $D: S \to \R^2$, which is a $(\Aff(\R^2),\R^2)$-local diffeomorphism, is actually a global diffeomorphism onto its image. 
		
$\bullet$ $D$ is a diffeomorphism on the orbits of $V$: since  $V$ parallel and complete, an orbit of $V$ is a complete geodesic.  So $D$ sends such an orbit on a geodesic of the affine space $\R^2$, i.e. on a line segment in $\R^2$. But we know that a local diffeomorphism on a manifold of dimension $1$ is actually a global diffeomorphism, so $D$ sends an orbit of $V$ diffeomorphically onto a line (a complete geodesic) in $\R^2$. In particular, $D$ sends the distribution defined by $V$ on a $1$-dimensional distribution of $\R^2$ of constant direction. So we can suppose that $D$ sends $V$ on the constant vector field $e_1=\partial_x$.
		
$\bullet$ For any maximal geodesic $\gamma$ transversal to $V$, the open subset $U$ of $S$ obtained by taking the image of $\gamma$ by the flow of $V$ is an injective domain for the developing map, hence isomorphic to the band $\R \times I_0 \subset \R^2$, where $I_0$ is an open interval of $\R$:  here again, $D$ sends $\gamma$ into a line segment in $\R^2$, hence sends $U$ bijectively to the saturation of that segment by the flow of $e_1$, i.e. to a band $\R \times I_0 \subset \R^2$. Another parametrization of $\gamma$ defines an isomorphism onto another band $\R^2 \times I_1$ of $\R^2$, affinely equivalent to $\R^2 \times I_0$. 
		
$\bullet$ If $U_1$ and $U_2$ are two open sets in $S$ as before, then $D(U_1 \cap U_2) = \R \times I_{1,2} \subset \R \times \R$, where   $I_{1,2}$ is an open interval of $\R$. 
		
The space of the leaves of $\mathcal{V}$ is a manifold of dimension $1$ diffeomorphic to $\R$. An atlas for this space is given by a countable set of geodesics $(\gamma_i)_{i \in \Z}$ transversal to $V$, such that $\gamma_i \cap \gamma_j \neq \emptyset$ if and only if $\vert i-j \vert =1$. Denote by $U_i$ the open set defined by taking the image of $\gamma_i$ by the flow of $V$. We get an atlas $(U_i, D_i:=D_{\vert U_i})$ on $S$ such that for any $i, j \in \Z$, $U_i \cap U_j \neq \emptyset$ if and only if $\vert i-j \vert =1$, and for all $ (x,y) \in D_{i+1}(U_i \cap U_{i+1})$,
$$ D_i \circ D_{i+1}^{-1}(x,y)=(x+\alpha_i, \lambda_i y + \beta_i), \lambda_i >0, \alpha_i, \beta_i \in \R.$$
So $S$ is diffeomorphic under $D$ to the quotient of $\coprod_{i \in \Z} \R \times I_i$ by $(x,y) \in C_{i+1} \sim f_i(x,y) \in C_i$, where $C_i := D_i(U_i \cap U_{i+1})$ and $f_i := D_i \circ D_{i+1}^{-1}$. Such a manifold is isomorphic to $\R \times I$, equipped with its affine structure, where $I$ is an open interval of $\R$ (that could be either $\R, ]-\infty , 0[, ]0, + \infty[$, or $]0, 1[$). 
\end{proof}
\begin{proof}[Proof of Fact \ref{Fact 2.12: S are flat bands}]
Denote by $\widetilde{S}$ the universal cover of $S$, and by $\widetilde{V}$ and $\widetilde{\mathcal{V}}$ the lifts of the corresponding objects to $\widetilde{S}$. 
It follows in particular from the previous lemma that a geodesic in $\widetilde{S}$ transversal to $\widetilde{V}$ cuts all the leaves of $\widetilde{\mathcal{V}}$. The surface $\widetilde{S}$ can be either Lorentzian or degenerate. Denote its metric by $\widetilde{g}$. In the first case, take $\gamma$ a null geodesic transversal to $\widetilde{V}$ such that $\widetilde{g}( \dot{\gamma},V ) = 1$. It cuts all the leaves of $\widetilde{\mathcal{V}}$. So the developing map in the previous lemma sends $\widetilde{S}$ diffeomorphically onto a band $\R \times I \subset \R^2$, such that $D_* \widetilde{V} = \partial_x$ and $(D^{-1})^* \widetilde{g} = 2 dxdy$. In the second case, take a geodesic $\gamma$ transversal to $\widetilde{V}$ such that $\widetilde{g}(\dot{\gamma}, \dot{\gamma}) = 1$. Again, $\gamma$ cuts all the leaves of $\widetilde{\mathcal{V}}$. Then $D$ sends $\widetilde{S}$ diffeomorphically onto a band $\R \times I \subset \R^2$, such that $D_* \widetilde{V} = \partial_x$ and $(D^{-1})^* \widetilde{g} = dy^2$. 
\end{proof}
	
\section{A synthetic proof of completeness of the  lightlike geodesic foliation}\label{Section: Completeness along the leaves}
	
In this section, we use the equivalence between the existence of $d$-dimensional totally geodesic submanifolds in $M$ and the integrability of some distribution on the Grassmann $d$-plane bundle $Gr_d(M)$ over $M$ (which extends the notion of a geodesic line field on the projective bundle) to define a $2$-dimensional foliation on some sub-bundle of $Gr_2(M)$. We will see that some properties of this foliation allow to study completeness properties of $M$.
	
The affine connection on $TM$ defines a connection on $Gr_2(M)$, so that the tangent bundle decomposes $T Gr_2(M) = H \oplus Vr $, where $Vr$ is the (canonical) vertical sub-bundle  and $H$ is the horizontal sub-bundle (determined by the connection).\\
	
\paragraph{T\textbf{he geodesic plane field on $Gr_2(M)$.}}   The geodesic plane field $\tau$ on $Gr_2(M)$ is defined for any $x \in M$, $p \in Gr_2(M)_x$ by the unique horizontal subspace $\tau_p \subset H_x$ (of dimension $2$) which projects on $p$ via $\dd_p \pi$.  This is equivalent to defining $\tau_p$ as  follows: consider a curve $x(t)$ in $M$ such that $x^{'}(0) \in p$, and parallel transport $p$ along $x(t)$. This defines a horizontal curve $\alpha(t)=(x(t), p(t))$ in $Gr_2(M)$, whose infinitesimal generator at $0$ gives a vector in $H_p$ that projects into $p$. The subspace $\tau_p$ is the one obtained by considering all such curves in $M$ tangent to $p$ at $0$. 
	
This defines a $2$-dimensional horizontal distribution $\tau$ on $Gr_2(M)$. A curve $\alpha(t)=(x(t),p(t)), t \in I$, is tangent to $\tau$ if:
\begin{itemize}
   \item $x^{'}(t) \in p$ for all $t \in I$,
   \item $p(t)$ is parallel along $x(t)$. 
\end{itemize}
	
\paragraph{\textbf{The leaves of $\tau$.}} A leaf $F$ of $\tau$ is a submanifold of dimension $2$ of $Gr_2(M)$ that projects into a plane field along a surface $S$ in $M$, which is parallel along any curve contained in $S$. This is equivalent to saying that there exists a surface $S$ in $M$, with $F=\{x \mapsto T_x S, x \in S\}$, and $S$ is totally geodesic. 
	
Let $p \in Gr_2(M)_x$. It results from the above that there exists a leaf of $\tau$ through $p$ if and only if $p$ is the tangent space at $x$ of a totally geodesic surface in $M$ through $x$, which in turn is equivalent to $\exp_x(p)$ being a totally geodesic submanifold of $M$ in a neighborhood of $x$. \\
	
\paragraph{\textbf{A domain of integrability of $\tau$.}} 
Let $(M,g,V)$ be a compact Brinkmann spacetime. 	
Consider the sub-bundle $X$ defined in Section \ref{Subsection: Existence of totally geodesic surfaces in a Brinkmann spacetime}, and let $p \in X$. Since $V$ is parallel, if there exists a leaf of $\tau$ through $p$, then it is entirely contained in $X$. By Proposition \ref{Proposition: existence of totally geodesic surfaces in M} and Fact \ref{Fact 2.12: S are flat bands}, we have the following fact.
\begin{fact}\label{Fact 3.1}
The distribution $\tau_{\vert X}$ is integrable, and a maximal leaf of $\tau$ through $p \in X$ projects under $\pi$ on a maximal totally geodesic flat surface $S_p$ in $M$ saturated by $\mathcal{V}$. Moreover, the universal cover of $S_p$ is a flat band.
\end{fact}
\noindent \textbf{Notation:} We denote by $\G$ the $2$-dimensional foliation on $X$ tangent to $\tau$.\\
	
\paragraph{\textbf{Completeness of the leaves of $\F$.}}
 \textit{Question}: do the leaves of $X$ project onto complete surfaces in $M$ ? We will see that the answer is yes when the projection is a surface contained in a leaf of $\F$. 
	
The sub-bundle $X$ is compact, as a bundle over a compact manifold with compact fiber  (the fiber of $X$ at $x$ is isomorphic to the projective space of  $T_x M / V$, hence to $\R P^{n-1}$).
Observe that for $p \in X$, the surface $S_p$ tangent to $p$ is contained in a leaf of $\mathcal{F}$ if and only if $p \subset V^{\perp}$. 
\medskip
	
Define $Y \subset X$ as follows:
$$ Y:=\{p \in X, p \subset V^{\perp}\}.$$ 
Then $Y$ is a closed submanifold of $X$ (hence compact), which is $\G$-invariant. A leaf of $Y$ projects into a surface in $M$ contained in a leaf of $\F$.
	
The flow of $V$ induces a flow on $Gr_2(M)$. Denote by $\widehat{V}$ the infinitesimal generator of this flow. Since $V$ is parallel, $\widehat{V}$ is also the unique horizontal vector field on $X$ which projects to $V$ under $\dd \pi$. 
	
\begin{fact}\label{Fact 3.2}
Let $\widehat{S}$ be a leaf of $\G$  and denote by $S$ its projection on $M$. If $\widehat{S}$ is contained in $Y$, then the universal cover of $S$ is an infinite flat band.
\end{fact} 
\begin{proof}
Let $\widehat{W}$ be a smooth vector field on $Y$ such that for any $p \in Y$, $\widehat{W}_p \in T_p \G$ and $\widehat{W}$ is transversal to $\widehat{V}$ (to  be more accurate, it is a direction field that we can always construct, and a vector field up to a double cover). By definition, $g(\dd_p \pi(\widehat{W}), V)=0$. Since $Y$ is compact, $\widehat{W}$ is complete. Furthermore, there is $\alpha >0$ such that $g(\dd_p \pi(\widehat{W}), \dd_p \pi(\widehat{W})) \geq \alpha >0$ for every $p \in Y$. Let $\widehat{S}$ be a leaf of $\G$. Since $\pi$ restricted to $\widehat{S}$ is a diffeomorphism onto its image, we can define $W_{S} := \dd \pi(\widehat{W}_{\vert \widehat{S}})$. Denote by $\widetilde{W_{S}}$ (resp. $\widetilde{V}$) the lift of $W_{S}$ (resp. $V$) to the universal cover $\widetilde{S}$. So $\widetilde{S}$ is a flat band foliated by $\widetilde{\mathcal{V}}$, and equipped with a vector field $\widetilde{W_{S}}$ transversal to $\widetilde{V}$. 
To see that it is an infinite band, let $c$ be an integral curve of $\widetilde{W_{S}}$ defined on $[0,T[$, it cuts all the leaves of $\widetilde{V}$. Set $l(c) = \int_{0}^{T} (\widetilde{g}_{c(t)}(\widetilde{W_{S}},\widetilde{W_{S}}))^{1/2} \dd t  $   the length of $c$.  We have $l(c) \geq \alpha^{1/2} \cdot T$. However, $\widehat{W}$ is complete, so $\widetilde{W_S}$ is complete, which yields $l(c) = \infty$. 
\end{proof}
\begin{corollary}
The leaves of $\F$ are complete.
\end{corollary}
\begin{proof}
Let $\gamma$ be a maximal geodesic tangent to $\F$. Let $S_p$ be the maximal totally geodesic surface in $M$ tangent to $p:=\mathsf{span}(\gamma'(0),V)$ (it exists by Fact \ref{Fact 3.1}). Then, by Fact \ref{Fact 3.2}, the universal cover of $S_p$ is an infinite flat band. Then the completeness of $\gamma$ is equivalent to the fact that $\gamma$ cuts all the leaves of $\widetilde{V}$, and this is true since the geodesics of a flat band are straight lines. 
\end{proof}

\section{The geodesic Equation}\label{Section: Geodesic equation}

Let $\gamma$ be a geodesic in $M$. Fix Rosen local coordinates $(u,v,x^1,\ldots,x^{n-1}) \in I \times J \times B_{n-1}(0,1)$, the geodesic equations are given by

\begin{align}\label{Equation geodesique: x(t)}
\ddot{x}^{k} + \Gamma^{k}_{ij}(u,x)\, \dot{x}^{i} \dot{x}^{j} + \Gamma_{iu}^{k}(u,x)\, \dot{x}^{i} \dot{u} =0\;\;\; \forall k \in \{1,\ldots,n-1\}
\end{align}
\begin{align}\label{Equation geodesique: v(t)}
\ddot{v}+ \Gamma_{ij}^{u}(u,x)\, \dot{x}^{i} \dot{x}^{j} + \Gamma_{iu}^{u}(u,x)\, \dot{x}^{i} \dot{u} =0,
\end{align}
\begin{align}\label{Equation geodesique: u(t)}
\ddot{u}=0,
\end{align}
where the $\Gamma_{ab}^c$'s are the Christoffel symbols of the Levi-Civita connection of the metric.

Since $V$ is parallel, $g(\dot{\gamma},V)$ is constant along $\gamma$. We distinguish two different situations: either $g(\dot{\gamma},V)= 0$ or  $g(\dot{\gamma},V) \neq 0$. In the first case, $\gamma$ is contained in a leaf of $\F$ and $u=u_0$ is constant along $\gamma$. In the second  case, one can suppose without loss of generality that $g(\dot{\gamma},V)=1$; this yields $\dot{u}=1$, hence $u=t$ (up to translation of the geodesic parameter). 

We see that when $\gamma$ is tangent to $\F$, the coefficients in the geodesic equations are time independent, and the previous equations become
\begin{align*}
&u=u_0,\\
&\ddot{x}^{k} + \Gamma^{k}_{ij}(u_0,x)\, \dot{x}^{i} \dot{x}^{j} =0\;\;\; \forall k \in \{1,\dots,n\},\\
&\ddot{v}+ \Gamma_{ij}^{u}(u_0,x)\, \dot{x}^{i} \dot{x}^{j}=0.
\end{align*}
They are autonomous equations. 

Denote by $h_{u}$ the Riemannian metric induced by $g$ on each slice $B_u:=\{u\} \times \{0\} \times B_{n-1}(0,1)$. The equation on $x(t)$ can be written  in the following way:

\begin{align}\label{Equation geodesique compacte: x(t)}
\nabla_{\dot{x}(t)}^t \dot{x}(t) = A(t,x) \cdot \dot{x}(t),
\end{align}
where $\nabla^t $ is the Levi-Civita connection of the Riemannian metric $h_{u(t)}$, where either $u(t)=u_0$ or $u(t)=t$, and $A(t,x) \in Gl_{n-1}(\R)$ is given by

$$
A_{ik}(t,x)= \left\{
\begin{array}{ll}
-\Gamma_{iu}^{k}(t,x)  & \mbox{if } \gamma \mbox{\;\;is not tangent to\;\;} \F \\
0 & \mbox{otherwise.}
\end{array}
\right.
$$

\section{Analysis of the geodesic equation}\label{Section: Analysis of geodesic equation}

Let $(M,g,V)$ be a compactly  Brinkmann-homogeneous manifold. Recall that this means that there is a compact subset $\mathcal{K}$ whose orbit 
by the Brinkmann isometry group $\Iso(M,g,V)$ covers all the space.	Observe that in this case, $V$ is a complete vector field. 
\medskip

Let $H$ be a smooth $(n-1)$-dimensional  distribution on $M$, such that for every $p \in M$, $H_p$ is a non-degenerate Riemannian subspace of $T_p M$ tangent to $\F$. This can be defined by taking an auxiliary Riemannian metric $g_0$, and setting $H:=V^{\perp_{g_0}} \cap V^{\perp_{g}}$.
For every $p \in M$,  there exists an open neighborhood $O_p$ of $0$ in $H_p$ such that the exponential map $\exp_p: H_p \subset T_p M \to M$ restricted to $H_p$ is a diffeomorphism from $O_p$ onto its image in $M$. Set $S_p := \exp_p (O_p)$ : it is an $(n-1)$-submanifold through $p$. 
\medskip

Let also $\mathsf{\textbf{Z}}$ be a smooth vector field on $M$ such that $g(\mathsf{\textbf{Z}},\mathsf{\textbf{Z}})=0$, $g(\mathsf{\textbf{Z}},V)=1$ and $\mathsf{\textbf{Z}} \in H^{\perp}$ (it is well defined up to taking a double cover of $M$). Both $H$ and $\mathsf{\textbf{Z}}$ are globally defined on $M$.

For every $p \in M$, consider a normal coordinate system $(x^1,\ldots,x^{n-1})$ on $S_p$. If $\rho = \sum (x^i)^2$, then for $r_p >0$ sufficiently small, $B_p(r_p) := \{q \in S_p, \rho(q) < r_p \}$ is a normal neighborhood of $p$ in $S_p$ diffeomorphic to $B_{n-1}(0,r_p)$, an open ball of $\R^{n-1}$ of center $0$ and radius $r_p$. 
Then, taking $r_p$ smaller if necessary, $p$ admits a Rosen coordinate chart $(u,v,x^1,\ldots,x^{n-1}) \in I_p \times J_p \times B_{n-1}(0,r_p)$, where  $I_p$ and $ J_p$ are intervals of the form $]-l_p,l_p[$,  $l_p>0$.  This chart is defined as in Fact \ref{Fact: Rosen coordinates} on an open neighborhood $U_p$ of $p$ by
\begin{align*}
f_p: &I_p \times J_p \times B_{n-1}(0,r_p) \to U_p \subset M \\
&(u,v,x) \mapsto \psi_u \circ \phi_v (\exp_{p_{|O_p}}(\sum_{i=1}^{n-1} x_i e_i)), \;\; e_i =\partial_{x_i},
\end{align*}
with $f(0,0,0)=p$, where $\phi_v$ is the flow of $V$, and $\psi_u$ is the local flow of a vector field $Z_{p} \in \Gamma(TU_p)$  which  coincides with $\mathsf{\textbf{Z}}$ on $f_p(\{0\}\times \{0\} \times B_{n-1}(0,r_0)) \subset U_p$ and satisfies $g(Z_{p},V)=1$ (see Fact \ref{Fact: Rosen coordinates}).  This new vector field $Z_{p}$ is a perturbation of $\mathsf{\textbf{Z}}$ in the neighborhood $U_p$ of $p$.

A continuity argument on the compact $\mathcal{K}$ shows that there exists $r_0>0$ such that for any $p \in \mathcal{K}$, we can take $B_{n-1}(0,r_p)$ above to be the open ball of radius $r_0$.  
Also, there exists a positive constant $l_0$ such that for every $p \in \mathcal{K}$, one can take $I_p$ and $J_p$ above to be the same interval of length $2 l_0$. 
\medskip

Henceforth, to each point $p \in \mathcal{K}$, we associate a Rosen coordinate chart $U_p \subset M$ as before. We can assume  without loss of generality that  $\frac{1}{2} l_0=r_0=1$. 
Each chart $U_p$ determines vector fields $X_1^p,\ldots,X_{n-1}^p, Z_{p}$ on $U_p$ such that 
\begin{align}\label{Eq: Z_m=Z+V+H}
   Z_{p} = \mathsf{\textbf{Z}} + a_p V + b_1^p X_1^p+\ldots+ b_{n-1}^p X_{n-1}^p, 
\end{align}
where $a_p, b_j^p \in C^{\infty}(U_p, \R)$. By compactness, we can assume that the functions $a_p, b_j^p$ are bounded on $U_p$, uniformly with respect to $p \in \mathcal{K}$. 
\bigskip

Define for $a, b, c \in \{u,v,1,\ldots,n-1\}$, the map $F_{(a,b,c)}: M \times I \times J \times B_{n-1}(0,1) \to \R$, $F_{(a,b,c)}(p, u, v, x) = (\Gamma_{ab}^{c})_p(u,x)$,
where $(\Gamma_{ab}^{c})_p$ is the Christoffel symbol associated to the pull-back of the metric $g_{\vert U_p}$ by the diffeomorphism $f_p$. Since $f_p$ depends smoothly on $p$, $F_{(a,b,c)}$ is a smooth map for every $a, b, c \in \{u,v,1,\ldots,n-1\}$. 
\bigskip

\textbf{Notation:} 
Denote by $h_u^{p}$ the Riemannian metric induced by
the pull-back of the metric $g_{\vert U_p}$ by $f_{p}$
on each slice $\{u\} \times \{0\} \times B_{n-1}(0,1)$.

\begin{lemma}\label{Lemme: comparaison des normes}
	Let $I=J=]-2,2[$ and 
	consider the space $D=I \times J \times B_{n-1}(0,1)$ equipped with a metric $g= 2 \displaystyle{d} u \displaystyle{d} v + g_{ij}(u,x) \displaystyle{d} x^i \displaystyle{d} x^j\; (u,v,x) \in D$. For every $u \in I$, the slice $\{u\} \times \{0\} \times B_{n-1}(0,1)$ inherits a Riemannian metric $h_u$ that depends smoothly on $u$. 	
	Let $c(t) = (u(t), v(t), x(t))$ be a geodesic in $D$, such that if $c(t)$ is not tangent to $\partial_v^\perp$, we set $g(\dot{c}(t),\partial_v)=1$. 
	Then, there exists $\epsilon >0$ and $C >0$ such that for any initial conditions $(u(0)=0, x(0)=0, \dot{x}(0) \in \R^{n-1})$ with  $h_0(\dot{x}(0), \dot{x}(0)) \geq 1$, $x(t)$ exists on $[0, \frac{\epsilon}{\parallel \dot{x(0)} \parallel}_0]$. And for all $t \in [0, \frac{\epsilon}{\parallel \dot{x}(0) \parallel}_0]$,
	$$ \frac{\parallel \dot{x}(t) \parallel_0}{\parallel \dot{x}(0) \parallel_0} \leq \sqrt{1+ Ct} ,$$
	and 
	$$ \frac{\parallel \dot{x}(t) \parallel_t}{\parallel \dot{x}(t) \parallel_0} \leq \sqrt{1+ Ct},$$
 which yields 
 $$ \frac{\parallel \dot{x}(t) \parallel_t}{\parallel \dot{x}(0) \parallel_0} \leq 1+ Ct,$$
	where $\parallel . \parallel_t$ is the norm for the metric $h_{u(t)}$. 
\end{lemma} 
\begin{proof}
The first inequality is the hardest to prove.
Consider in the  ball $B_{n-1}(0,1)$ of $\R^{n-1}$:
\begin{align}\label{Lemma: non_autonomeous_equation}
\ddot{x}^k = -  \Gamma_{ij}^k (u, x) \dot{x}^i \dot{x}^j + A_{ik}(u, x) \dot{x}^i,
\end{align}
for $k=1,\ldots,n-1$, with initial conditions $x(0) =0$, and $\dot{x}(0) \in \R^{n-1}$ such that $h_0(\dot{x}(0), \dot{x}(0)) \geq 1$. 	
If $c(t)$ is tangent to $\partial_v^\perp$, then $u(t)$ is identically zero (we assumed $u(c(0))=0$). Otherwise, $u(t)=t$. 

Consider $\alpha(t) := c(\lambda^{-1} t)$, with $\lambda:= \parallel \dot{x}(0) \parallel_0$, and its component $y(t):=x(\alpha(t))=x(\lambda^{-1} t)$. Then $u(\alpha(t))=\lambda^{-1} u$, since $u(t)$ is a linear function of $t$. And $y(t)$ satisfies the equation
\begin{align}\label{Lemma: normalization}
\ddot{y}^k (t)  = - \Gamma_{ij}^k ( \lambda^{-1}  u,  y) \dot{y}^i \dot{y}^j +  \lambda^{-1} A_{ik}( \lambda^{-1} u, y) \dot{y}^i,
\end{align}
with initial conditions $y(0) =0$ and $h_0(\dot{y}(0), \dot{y}(0))=1$. 
\bigskip

Write $\Gamma_{ij}^k ( \lambda^{-1}  u,  y) = \Gamma_{ij}^k (0,  y) + \lambda^{-1} u F_{ij}^k( \lambda^{-1}u , y)$, where 
$F_{ij}^k$ are continuous and $u \in [0, 1]$.

Consider 
$$\frac{\partial}{\partial t}h_0 (\dot{y}, \dot{y}) = 2 h_0(\nabla^0_{\dot{y}}\dot{y}, \dot{y}). $$

But $$ \nabla^0_{\dot{y}}\dot{y} =  \big( \ddot{ y}^k (t) + \Gamma_{ij}^k (0,  y) \dot{y}^i \dot{y}^j \big) \partial_k, \;\hbox{with}\; \partial_k = \frac{\partial}{\partial y^k}.$$

Thus, from Equation (\ref{Lemma: normalization}), 
$$  \nabla^0_{\dot{y}}\dot{y}  =  \frac{1}{\lambda} \big( u F_{ij}^k( \lambda^{-1}u , y) \dot{y}^i \dot{y}^j +   A_{ik}( \lambda^{-1} u  , y) \dot{y}^i  \big) \partial_k,$$ 

hence
\begin{align}\label{Lemma: Equation differentielle sur la norme}
\frac{\partial}{\partial t}h_0 (\dot{y}, \dot{y}) =  \frac{2}{\lambda} h_0 ( ( u F_{ij}^k( \lambda^{-1}u, y) \dot{y}^i \dot{y}^j +   A_{ik}( \lambda^{-1} u, y) \dot{y}^i ) \partial_k, \dot{y} ).
\end{align}

 \bigskip 
There exists  $\epsilon >0$ such that for any initial conditions $(y(0)=0, \dot{y}(0) \in \R^{n-1})$, with $h_0(\dot{y}(0), \dot{y}(0))=1$, $y(t)$ is defined on $[0,\epsilon]$ and $\dot{y}(t)$ has a uniformly bounded $h_0$ norm. 
We will use Equation (\ref{Lemma: Equation differentielle sur la norme}) to get a finer control of the $h_0$ norm of $\dot{y}(t)$ for $t \in [0, \epsilon]$. A continuity argument on the coefficients involved in the equation ensures the existence of constants $\epsilon' >0$ and $r_0>0$ such that all the coefficients are bounded on $[0, \epsilon'] \times I \times B_{n-1}(0,r_0)$. 

We denote again $\epsilon$ instead of $\epsilon'$.
Therefore, there is a constant $C>0$ such that:
$$2h_0 \big( \big( u F_{ij}^k( \lambda^{-1}u , y) \dot{y}^i \dot{y}^j +   A_{ik}( \lambda^{-1} u , y) \dot{y}^i, y)  \big) \partial_k,  \dot{y} \big) \leq C, \;t \in [0, \epsilon].$$
Hence
$$\frac{\partial}{\partial t}h_0 (\dot{y}, \dot{y}) \leq  \frac{C}{\lambda}, t\in [0, \epsilon]. $$
Thus $$h_0(\dot{y}(t), \dot{y}(t)) \leq  \frac{C}{\lambda} t + 1, t \in [0, \epsilon] \;  (\hbox{since} \; h_0(\dot{y}(0), \dot{y}(0) = 1). $$ 
This gives
\begin{align}
\sqrt{ \frac{h_0(\dot{x}(t), \dot{x}(t))}{h_0(\dot{x}(0), \dot{x}(0))}  }\leq   \sqrt{1+ Ct },\;\; t \in [0,  \frac{\epsilon}{\lambda}],
\end{align}
which yields the first inequality.\\

Let us now assume $u=t$, and compare the $h_{t}$ and the $h_0$-norms of  $\dot{x}(t)$, for $t \in   [0,  \frac{\epsilon}{\parallel  \dot{x}(0)\parallel}_0]$.
Consider the function   $l(t,w) :=  \frac{h_t(w, w)}{h_0(w, w)}$, for $t \in [0,1]$ and $w \in \R^{n-1}$. Observe that $l(t,\alpha w) = l(t,w)$, hence 
we can assume $h_0(w, w) = 1$. By smoothness and compactness,    $\frac{\vert l(t,w) - l(0,w) \vert}{t}  = \frac{\vert l(t,w) - 1) \vert}{t}
$ is bounded by some constant, say the same $C$ as above. 
So 
$$ \sqrt{\frac{h_t(\dot{x}(t) ,  \dot{x}(t))}{h_0( \dot{x}(t),  \dot{x}(t))} } \leq  \sqrt{1 + Ct}.$$
\end{proof}

\begin{lemma}\label{Lemme: la geodesique definie tant que x((t) est definie)}
Let $\gamma$ be a geodesic of $M$, and set $p:=\gamma(0)$. If $\gamma$ is not tangent to $\F$, we parameterize in such a way that  $g(\dot{\gamma}, V)=1$. Let $(u,v,x)$ be the system of coordinates on the Rosen chart $f^{-1}_p(U_p)=I \times J \times B_{n-1}(0,1)$. 
If $x(t)$ is defined on $[0,T]$, then  $\gamma(t)$ is defined on $[0,T']$, with $T':= \inf(1, T)$. 
\end{lemma}
\begin{proof}	
It follows from Equation (\ref{Equation geodesique: u(t)}) on $u(t)$ that either $u(t) \equiv 0$ or $u(t)= t$.  Consider the interval $[0,T]$ on which $x(t)$ is defined. Then $u(t)$ is defined on the same interval, and $\vert u(T) - u(0) \vert < \frac{1}{2}\vert I \vert$, where $\vert I \vert$ denotes the length of $I$. Then on $[0, T']$, both the $x(t)$ and the $u(t)$ components are defined. 

It appears from Equation (\ref{Equation geodesique: v(t)}) that  $v(t)$  is defined on  the same interval as $x(t)$. Set $l:=\vert v(0) - v(T) \vert$.  However,  $v(t)$ may well leave the interval $J$ before $x(t)$ leaves the ball $B_{n-1}(0,1)$, so that $\gamma_{\vert U_{p}}$ is not defined on $[0,T']$. To prove that $\gamma$ is defined on $[0,T']$, we will use the fact that the flow of $V$ acts isometrically on $M$. Denote by $\phi^s$ the flow of $V$. Set $R:= \frac{1}{2}\vert J \vert$. 

Assume that $l \geq R$. Let $q_0:=\gamma(\alpha_0) \in U_{p}$ such that $\frac{R}{2}<s_0:=v(\alpha_0)<R$.  The translate $W_1:=\phi^{s_0}(U_{p})$ of the Rosen chart at $p$ is a Rosen chart at $\phi^{s_0}(p)$. We denote by $(u_1,v_1,x_1):=(u,v,x) \circ \phi^{-s_0}$ the determined Rosen coordinates in the new chart. Consider the geodesic $\gamma_1$ such that $\gamma_1(0)=q_0$ and $\gamma_1'(0)=\gamma'(\alpha_0)$.
Then $u_1(t)$, $v_1(t)$ and $x_1(t)$ are defined on $[0, T'-\alpha_0]$, with $v_1(t)=v(t+\alpha_0)-v(\alpha_0)$.
Since by definition $R - v(\alpha_0) < \frac{R}{2}$, the geodesic $\gamma_1$ extends $\gamma_{\vert U_{p}}$ on $W_1$, beyond $\overline{U_{p}}$. 
If $l- v(\alpha_0) <R$, then $\vert v_1(T'-\alpha_0)-v_1(0) \vert <R$. In this case, all the components of $\gamma_1$ are defined on $[0, T'-\alpha_0]$, and $\gamma_1$ extends $\gamma_{\vert U_{p}}$, which proves that $\gamma$ is defined on $[0,T']$. If not, 
consider  $q_1:=\gamma(\alpha_1) \in W_1 \smallsetminus \overline{U_{p}}$ such that $s_0+\frac{R}{2}<s_1:=v(\alpha_1)<s_0+R$. Doing this, we obtain a  sequence $q_0, q_1,\ldots,q_i,\ldots$ such that $q_{i+1}:= \gamma(\alpha_{i+1}) \in   W_{i+1} \smallsetminus \overline{W_i}$ and $s_i+\frac{R}{2}<s_{i+1}:=v(\alpha_{i+1})<s_i+R$. And we have $v(\alpha_i) > \frac{(i+1)R}{2}$. 
At some moment, we will have $l - v(\alpha_n) <R$, it suffices to take $n:=\lfloor 2 (\frac{l}{R} -1) \rfloor +1$. So the sequence above is finite. At each step, we can then extend $\gamma_{\vert U_p \cup \bigcup_{j=0}^i W_j}$ as before on $W_{i+1}$. The condition $l - v(\alpha_n) <R$ ensures that the process will stop at $W_n$, and $\gamma_{\vert U_p \cup \bigcup_{j=0}^{n-1} W_j}$ extends on $W_n$ to be defined on $[0,T']$. 
\end{proof}
\begin{theorem}\label{theorem-main: compact Brinkmann is complete} 
A compactly Brinkmann-homogeneous spacetime is complete.
\end{theorem}

\begin{proof}
\underline{Step 1:}  Let $\gamma$ be a geodesic in $M$. If $\gamma$ is not tangent to $\mathcal{F}$, we parameterize in such a way that $g(\dot{\gamma}, V) = 1)$.  Assume $p:=\gamma(0) \notin \mathcal{K}$. By assumption, there exists a Brinkmann isometry, i.e. $\varphi \in \Iso(M,g,V)$, such that the new geodesic $\alpha(t):=\varphi(\gamma(t))$ satisfies 
$\alpha(0) \in \mathcal{K}$ and $g(\dot{\alpha},V)=1$. We write the geodesic equation for $\alpha(t)$ in the Rosen chart $U_{\varphi(p)}$ associated to $\varphi(p) \in \mathcal{K}$. Since $g(\dot{\alpha},V)=1$, we can apply Lemma \ref{Lemme: comparaison des normes} to estimate the time existence of $\alpha(t)$, which is the same for $\gamma(t)$. We do this systematically, when we deal with a point not contained in $\mathcal{K}$. 
\medskip

Since we will use different Rosen charts, we will write $(u_q,v_q,x_q)$ for the Rosen coordinates of a Rosen chart $f_q^{-1}(U_q)$. We have $(u_q,v_q,x_q)(q)=(0,0,0)$.
\medskip

So let $p \in \mathcal{K}$, and $\gamma$ such that $\gamma(0)=p$. Write $(u_p(t),v_p(t),x_p(t))=(u_p,v_p,x_p)(\gamma(t))$.

The geodesic equation in the Rosen chart $f_p^{-1}(U_p)$ gives  the following equations in the  ball $B_{n-1}(0,1)$ of $\R^{n-1}$:
\begin{align}\label{non_autonomeous_equation}
	\ddot{x}^k_p = -  \Gamma_{ij}^k (u, x_p) \dot{x}^i_p \dot{x}^j_p + A_{ik}(u, x_p) \dot{x}^i_p,
\end{align}
for $k=1,\ldots,n-1$, with initial conditions $(x_p(0) =0, \dot{x}_p(0) \in \R^{n-1})$, and either $u_p(t)\equiv 0$ or $u_p(t)=t$. \\

We will use the same notation $\parallel . \parallel_t$ introduced in Lemma \ref{Lemme: comparaison des normes}. When we write $\parallel \dot{x}_p(t) \parallel_{t}$, this  means that the vector $\dot{\gamma}(t)$ is decomposed in the basis $(Z_{p}, V, X_1^{p},\ldots,X_{n-1}^{p})$ associated to $U_{p}$, and we compute the $h^{p}_{u(t)}$-norm of its $x_p$-component in this basis.
\medskip

$\bullet$ There exists $\epsilon >0$ such that for any $p \in \mathcal{K}$, and any initial condition $(x_p(0)=0,\dot{x}_p(0) \in \R^{n-1})$ such that $\parallel \dot{x}_p(0)	\parallel_{0} \leq 1$, the solution $x_p(t)$ is defined on $[0, \epsilon]$.  Indeed, since $\mathcal{K}$ is compact, the subset of vectors $X \in T_0 \R^{n-1}$ such that $h_0^p(X,X) \leq 1$ for some $p\in \mathcal{K}$, is a compact set in $T_0 \R^{n-1}$, hence, the choice of $\epsilon$ is uniform with respect to this data.  \\

$\bullet$ On the other hand, by Lemma \ref{Lemme: comparaison des normes},  there exist $\epsilon' >0$ and $C>0$ such that for any initial conditions $(x_p(0)=0, \dot{x}_p(0) \in \R^{n-1})$ with $ \parallel \dot{x}_p(0) \parallel_{0} \geq 1$, $x_p(t)$ exists on $[0, \frac{\epsilon'}{\parallel \dot{x}_p(0)\parallel_{0}}]$, and 
\begin{align}
\forall t \in [0,  \frac{\epsilon'}{\parallel  \dot{x}_p(0)\parallel_{0}}], \;\; \frac{  \parallel \dot{x}_p(t) \parallel_{t}  }{ \parallel \dot{x}_p(0) \parallel_{0} }  \leq   1+ Ct.
\end{align}

This gives for $  t \in [0,  \frac{\epsilon'}{\parallel  \dot{x}_p(0)\parallel_{0}}]$, 
\begin{align*}
 \parallel \dot{x}_p (t) \parallel_{t}  &\leq \parallel \dot{x}_p(0) \parallel_{0} (1+ Ct)   \\
 & \leq \parallel \dot{x}_p(0) \parallel_{0} +  
C\parallel\dot{x}_p(0)\parallel_{0} \frac{\epsilon'}{\parallel \dot{x}_p(0) \parallel_{0}} \\
 &= \parallel \dot{x}_p (0)  \parallel_{0} + \epsilon' C.
\end{align*}

If we look again at the proof of Lemma  \ref{Lemme: comparaison des normes}, it follows from the compactness of $\mathcal{K}$ and the smoothness of the maps $F_{(a,b,c)}$ defined at the beginning of this section that the choices of $\epsilon'$ and $C$ are actually uniform with respect to $p \in \mathcal{K}$. \\

\noindent \textbf{Notation:} 
From now on, and since there will be no confusion, we will denote simply $\parallel \dot{x}_p(t) \parallel:=\parallel \dot{x}_p(t) \parallel_t$, for $p \in \mathcal{K}$.
\bigskip

\underline{Step 2:} Set $s_0=0$ and $q_0:=\gamma(s_0)$. 
By Step 1, $\gamma(t)=(u_{q_0}(t),v_{q_0}(t),x_{q_0}(t))$ is defined on $[0,s_1]$, with $s_1 = \inf(1,\inf(\epsilon, \frac{\epsilon'}{\parallel \dot{x}_{q_0}(0) \parallel}))$. And, 
$$\parallel \dot{x}_{q_0}(s_1) \parallel \leq \parallel \dot{x}_{q_0}(0) \parallel + \epsilon' C.$$
\noindent Set $q_1 := \gamma(s_1)$.

Next, set $\alpha(t):=\gamma(t+s_1)$, and put $\alpha(t)=(u_{q_1}(t),v_{q_1}(t),x_{q_1}(t))$ in the local chart $f_{q_1}^{-1}(U_{q_1})$. Write Equation (\ref{Equation geodesique: x(t)}) on $f_{q_1}^{-1}(U_{q_1})$ with initial conditions $x_{q_1}(0)=0$ and $\dot{x}_{q_1}(0) \in \R^{n-1}$ equal to the $x_{q_1}$-component of $\dot{\gamma}(s_1)$ in the frame associated to the local chart $U_{q_1}$. Then, $\gamma(t+s_1)$ is defined on $[0, s_2]$, with $s_2= \inf(1, \epsilon, \frac{\epsilon'}{\parallel \dot{x}_{q_1}(0) \parallel})$.
\bigskip

$\bullet$ Compare between $\parallel \dot{x}_{q_0}(s_1) \parallel$ and $\parallel \dot{x}_{q_1}(0) \parallel$ (here, we compare between the norms of the $x$-components of $\dot{\gamma}(s_1)$ in the two frames associated respectively to $U_{q_0}$ and $U_{q_1}$).

From (\ref{Eq: Z_m=Z+V+H}) we have that 
\begin{align*}
  Z_{q_0}&=\mathsf{\textbf{Z}}+a_{q_0} V + b_1^{q_0} X_1^{q_0} + \ldots +b_{n-1}^{q_0} X_{n-1}^{q_0} \\
  Z_{q_1}&=\mathsf{\textbf{Z}}+a_{q_1} V + b_1^{q_1} X_1^{q_1} + \ldots +b_{n-1}^{q_1} X_{n-1}^{q_1},
\end{align*}
where $a_{q_i}$ and $b_j^{q_i}$ are bounded functions. 
So on $U_{q_0} \cap U_{q_1}$ we can write for $i \in \{0,1\}$: 
\begin{align*}
    \dot{\gamma}(t+s_i)&= Z_{q_i} + \dot{v}_{q_i}(t) V + \dot{x}_{q_i}^1(t) X_1^{q_i}+ \ldots + \dot{x}_{q_i}^{n-1}(t) X_{n-1}^{q_i}.
\end{align*}
which gives 
\begin{align}\label{Eq: compare x-norms in different frames}
  \dot{\gamma}(t+s_i)-\mathsf{\textbf{Z}} = (a_{q_i}+\dot{v}_{q_i}(t)) V + \sum_{j=1}^{n-1} b_j^{q_i}  X_j^{q_i}  + \sum_{j=1}^{n-1} \dot{x}_{q_i}^j(t) X_j^{q_i}. 
\end{align}
Computing the norms for the Lorentzian metric $g$ in both sides of Equation (\ref{Eq: compare x-norms in different frames}), at the point $t=s_1$ when $i=0$, and at $t=0$ when $i=1$, 
and using the triangle inequality, we obtain  
\begin{align*}
    &\vert \parallel \dot{x}_{q_0}(s_1) \parallel - g(\dot{\gamma}(s_1)-\mathsf{\textbf{Z}}, \dot{\gamma}(s_1)-\mathsf{\textbf{Z}})^{\frac{1}{2}} \vert \leq  g\left( \sum_{j=1}^{n-1} b_j^{q_0}  X_j^{q_0} , \sum_{j=1}^{n-1} b_j^{q_0}  X_j^{q_0}\right)^{\frac{1}{2}} \\
    & \vert \parallel \dot{x}_{q_1}(0) \parallel - g(\dot{\gamma}(s_1)-\mathsf{\textbf{Z}}, \dot{\gamma}(s_1)-\mathsf{\textbf{Z}})^{\frac{1}{2}} \vert \leq g\left( \sum_{j=1}^{n-1} b_j^{q_1}  X_j^{q_1} , \sum_{j=1}^{n-1} b_j^{q_1}  X_j^{q_1} \right)^{\frac{1}{2}}
\end{align*}
which yields the existence of a constant $B>0$  such that 
\begin{align*}
    \vert \parallel \dot{x}_{q_0}(s_1) \parallel - \parallel \dot{x}_{q_1}(0) \parallel \vert \leq B.
\end{align*}
This constant $B>0$ is independent of $s_1$. 

\bigskip

\underline{Concluding:} We get a sequence $\gamma(t_n)$, such that $t_n:= \sum_{i=0}^n s_i$, where $s_i$ is\\

\noindent - either equal to $\epsilon$ or to $1$ for infinitely many $i$, \\
\noindent - or $s_{i}= \frac{\epsilon' }{ \parallel \dot{x}_{q_{i-1}}(0) \parallel}$, for all but a finite number of $i$, with the following inequalities
\begin{align*}
& \vert \parallel \dot{x}_{q_{i}}(s_{i+1}) \parallel - \parallel \dot{x}_{q_{i+1}}(0) \parallel \vert \leq B   \\
&  \parallel\dot{x}_{q_i}(s_{i+1})\parallel   \leq  
\parallel\dot{x}_{q_i}(0)\parallel  + \epsilon' C.
\end{align*}
In this case,  the sequence 
$\parallel \dot{x}_{q_i}(0) \parallel $ is  $(B+ \epsilon' C)$)-dense in the interval  $[\parallel\dot{x}_{q_0}(0)\parallel,  + \infty[$, hence
$\Sigma \frac{\epsilon' }{ \parallel \dot{x}_{q_i}(0) \parallel} = \infty$. \\

So in all cases, $\sum s_i = \infty$, and $\gamma$ is then complete.
\end{proof}

\section{Proof of Theorem \ref{reduction}, An Adapted Cartan structure}\label{Cartan_Connection}
	
The aim of the present section is to prove the existence of a core  as in  Theorem \ref{reduction}, the algebraic description which completes the proof will be given in the next section.  The exposition is  largely inspired by  Sections  4  and 5 of \cite{Fra},  and also by \cite{Pec} which  provides a natural and efficient approach to Gromov's theory on rigid transformation groups (see \cite{Gro}) via Cartan connections.

\subsection{Null frames}  Let $(E, \langle, \rangle)$ be a linear Lorentzian space of dimension $n+1$. A frame $(e_0, \ldots, e_n)$ is said to be orthonormal if all the vectors $e_i$ are orthogonal and $1 = -  \langle e_0,  e_0\rangle = + \langle e_1,  e_1\rangle = \ldots =\langle e_n,  e_n\rangle$.  This frame is said to be null if all Lorentzian products vanish but $1 = \langle e_0,  e_n\rangle = \langle e_1, e_1 \rangle = \ldots  = \langle e_{n-1},  e_{n-1}\rangle$ (in particular $e_0$ and $e_n$ are lightlike).  There is a simple correspondence between orthonormal and null frames, mainly, if $(e_0, \ldots, e_n)$ is orthonormal, then   $(\frac{e_0 + e_n}{\sqrt{2}},  e_2, \ldots, e_{n-1},  \frac{- e_0  + e_n}{\sqrt{2}}) $ is a null frame. 
	
Let $(M, g)$ be a Lorentzian manifold. Classical connections, as well as Cartan connections,  are usually developed on the orthonormal frame bundle, but there is no harm to consider  null frames instead. So let $\widehat{M}$ be the space of null frames (in all tangent spaces of $M$).  This is a $\O(1, n)$-principal bundle ($n +1 = \dim M $). 
	
\subsection{Adapted Cartan structure}
\subsubsection{Construction}
Let $V$ be a (non-singular) null vector field on $M$ and consider $\widehat{V} \subset \widehat{M}$ be those null frames starting with $V$. This is a $H$-principal bundle, where $H$ is the subgroup of $\O(1, n)$ preserving a lightlike vector. 
	
Consider the  Lorentzian quadratic form $Q = x_0 x_n + x_1^2 + \ldots + x_{n-1}^2$, for which the canonical basis is null. Elements of the orthogonal group $\O(Q)$ preserving $e_0$ have matrices of the form: 
\begin{equation*}
\begin{pmatrix}
			1 & \lambda_{1} & \cdots & \lambda_{n-1} & * \\
			0 & \alpha_{1,1} & \cdots & \alpha_{1,n-1} & * \\
			\vdots  & \vdots  & \ddots & \vdots & \vdots \\
			0 & \alpha_{n-1,1} & \cdots & \alpha_{n-1,n-1} & * \\
			0 & 0 & \cdots & 0 & 1
\end{pmatrix}
\end{equation*}
where $\lambda= (\lambda_1,\ldots,\lambda_{n-1}) \in \R^{n-1}$ and $\alpha= (\alpha_{ij})_{1 \leq i,j \leq n-1} \in \SO(n-1)$, and the $*$-entries  are completely determined by $\lambda$ and $\alpha$.

It follows in particular that $H $ is isomorphic to $\SO(n-1) \ltimes \R^{n-1}$, which is  the Euclide group $\mathsf{Euc}_{n-1}$, the (affine)  isometry group of the Euclidean space $\R^{n-1}$. 
	
The Levi-Civita connection associated to $g$ gives rise to a bundle connection on $\widehat{M}$,  i.e. a horizontal $\O(1, n)$-invariant distribution $\H$. This yields a ``tautological'' parallelism of $T\widehat{M}$. From the point of view of Cartan connections, this parallelism is expressed as a vectorial  differential 1-form $\omega$ which establishes, for any $\widehat{x}\in \widehat{M}$,  an isomorphism $\omega_{\widehat{x}}: T_{\widehat{x}} \textcolor{red}{\widehat{M}} \to \mathfrak{po}$, where $\mathfrak{po}$ is the Lie algebra of the Poincar\'e group $\mathsf{Poi}_{1, n} =  \O(1, n) \ltimes \R^{n+1}$. Furthermore, $\omega$ is $\O(1, n)$-equivariant, and sends the tangent space of fibers identically  to $\mathfrak{o}(1, n) \subset \mathfrak{po}$. We also have that the horizontal $\H$ is sent by $\omega$ to $\R^{n+1}$, see \cite{Pec, Fra}.

\begin{fact} The horizontal $\H$ is tangent to $\widehat{V}$ iff $V$ is parallel, iff,  $\omega$ sends $\widehat{V}$ to the Lie subalgebra $ \h \ltimes \R^{n+1} \subset \mathfrak{po}$.
\end{fact}

\begin{proof} 
		Let $X \in \H_u, u \in \widehat{V}$, and $\gamma$ a curve in $M$ such that $\gamma(0) = \pi(u)$ and $\gamma'(0)= d_u \pi(X)$. The horizontal lift of $\gamma$ starting at $u$ is the curve $u(t)$ such that $u(t)$ is a parallel frame field along $\gamma$ with $u(0)=u$ and $u'(0)=X$. Then $V$ is parallel along $\gamma$ if and only if the horizontal curve $u(t) \in \widehat{V}$ for every $t$. It follows that  if $V$ is parallel, then  $X \in T_u \widehat{V}$, for every $X \in  \H_u $. Now, suppose that the horizontal distribution $\H$ is tangent to $\widehat{V}$, then for every curve $\gamma$, the horizontal lift $u(t)$ starting at $u \in \widehat{V}$ is everywhere tangent to $\widehat{V}$, and hence contained in $\widehat{V}$, proving that $V$ is parallel. 
		
		On the other hand, $\omega_u^{-1}(\{0\} \times \R^{n+1})$ is a horizontal subspace of $T_u \widehat{M}$ of dimension $n+1$, hence  $\omega_u^{-1}(\{0\} \times \R^{n+1}) = \H_u$. Knowing that $\widehat{V}$ is an $H$-principal bundle on $M$, it follows that $\omega$ sends $\widehat{V}$ to $\h \ltimes \R^{n+1}$ if and only if $\H$ is tangent to $\widehat{V}$. 
\end{proof}
\medskip
	
Henceforth, we will assume $V$ a parallel vector field. 
\medskip
  
In a different but related manner,  parallel submanifolds of $\widehat{M}$ are defined in   \cite{Pec} by the fact that $\omega$ send all their tangent spaces to the same subspace of $\mathfrak{po}$. This is the case of $\widehat{V}$. In fact $(\widehat{V}, M, \omega_{\vert \widehat{V}})$ is a Cartan geometry, say,  a sub Cartan geometry of $(\widehat{M}, M, \omega)$ in  a natural sense.

\subsubsection{Curvatures}
The curvature  of  the Cartan  connection $\omega$ is defined by the Cartan-Maurer formula $\Omega=  d \omega  + 1/2 [\omega, \omega]$. Observe that the curvature of $\widehat{V}$ is just the restriction of that of $\widehat{M}$.  
	
Using the parallelism, the $\widehat{M}$ curvature is encoded in a $\O(1, n)$-equivariant map 
$$\kappa: \widehat{M} \to \mathcal W_0 = \mathsf{Hom} (\wedge^2 (\mathfrak{po}/\mathfrak{0}(1, n)), \mathfrak{po}) = \mathsf{Hom} (\wedge^2 \R^{n+1}, \mathfrak{po}).$$
The associated curvature map $\kappa_{\widehat{V}}$ for $\widehat{V}$ is just the restriction of $\kappa$. It is $H$-equivariant and takes     values in $ \mathcal W_0^{\widehat{V}} = \mathsf{Hom} (\wedge^2 \R^{n+1}, \h\ltimes \R^{n+1})$. 
	
Similarly, the parallelism allows one to express the differential $d \kappa$ as a vectorial map $\mathcal D^1 \kappa: \widehat{M} \to \mathsf{Hom}(\mathfrak{po}, \mathcal W_0) = \mathcal W_1$, and in  the same way, we have $\mathcal D^1 \kappa_{\widehat{V}}: \widehat{M} \to \mathsf{Hom}(\h \ltimes \R^{n+1}, \mathcal W_0^{\widehat{V}}) = \mathcal W_1^{\widehat{V}}$. One defines inductively $\mathcal D^i \kappa $, $\mathcal D^i \kappa_{\widehat{V}}$, $\mathcal W_i$ and $\mathcal W_i^{\widehat{V}}$, for $i >1$.

	\bigskip

	To put all these curvatures together, for any $l$, 
	consider $\kappa_l=(\kappa, D^1 \kappa,\ldots,D^l \kappa): \widehat{M} \to \mathcal W_0 \times \mathcal W_1 \times \ldots \times \mathcal W_l  =  \mathcal Y_l$, and   analogously, $\kappa_l^{\widehat{V}}=(\kappa_{\widehat{V}}, D^1 \kappa_{\widehat{V}},\ldots,D^l \kappa_{\widehat{V}}): 
	\widehat{M} \to \mathcal W_0^{\widehat{V}} \times \mathcal W_1^{\widehat{V}} \times \ldots \times \mathcal W_l^{\widehat{V}} = \mathcal Y_l^{\widehat{V}}$. 
	
	\subsubsection{Pseudo-group  (resp. pseudo-algebra) of local (resp. infinitesimal)  isometries}  A local isometry of $(M, g)$ is a triplet 
	$(U_1, U_2, f)$ where $U_1$ and $U_2$ are open in $M$ and $f$ is an isometry $U_1 \to U_2$
	(for the restricted metrics).  Here, we will restrict ourselves to local isometries preserving the 
	vector field $V$. Define a local Killing field as a pair $(U, X)$ where $U$ is open and 
	$X$ is a Killing field defined on $U$. Again, we will assume that $X$ commutes with $V$.

	Let $\G$ be the collection  of all these local isometries, and $\g$  that of all these 
	local Killing fields. 
	Let also $\G^0 \subset \G$ be  the subset of local isometries given by composition of local flows of local Killing fields. It 
	plays the role of the identity component of $\G$. 
	
	It is quite delicate, and this is not our purpose here,  to formulate all these concepts (of pseudo-groups and pseudo-Lie algebras), but let us observe that 
	it is straightforward to define orbits of $\G$ and also of $\G^0$ (see for instance \cite{Hae} and related references for 
	a rigorous treatment of pseudo-groups). One can in particular see that $\G^0$ has 
	connected orbits. One can also use $\G$ to define locally homogeneous spaces by the fact
	that $\G$ has one orbit.
	
	Observe here that $\G$ acts locally on $\widehat{M}$, by preserving $\widehat{V}$, the Cartan connection $\omega$, 
	its restriction  $\omega_{\vert \widetilde{V}} $,  and therefore all the previously defined
	curvature maps.

\subsection{A closed partition $\mathcal P_l$  of $M$} Fix $l$ (to be chosen later, big enough). The $\kappa_l^{\widehat{V}}$-levels determine a partition, say $\widehat{\mathcal P}_l$ into closed subsets. This projects to a partition $\mathcal P_l$ of $M$, but no longer by closed subsets, a priori. If we want the $\mathcal P_l$-parts to be closed in $M$, we have to  be sure that the $H$-saturation of a $\widehat{ \mathcal P}_l $-part is closed in $\widehat{M}$. But this is 
nothing but the $ \kappa_l^{\widehat{V}}$-inverse image of an $H$-orbit in $\mathcal Y_l^{\widehat{V}}$. This is guaranteed by:

\begin{fact} The $H$-orbits in $\mathcal Y_l^{\widehat{V}}$ are closed. 
\end{fact}
	
\begin{proof} Remember $H $ is a semi-direct product $\SO(n-1) \ltimes \R^{n-1} $. It is enough to show closeness of $N$-orbits, where $N := \R^{n-1}$ is the nil-radical of $H$. Let us now observe that  the $N$-representation in $ \mathcal Y_l^{\widehat{V}}$ is unipotent. This can be  shown explicitly, or  deduced from the fact that this is nothing but the  restriction  to $N$ of the $\O(1, n)$-representation in $\mathcal Y_l$. Now, $N$ is unipotent in $\O(1, n)$ and hence, because $\O(1, n)$ is semisimple, for any representation of $\O(1, n)$, $N$ acts unipotently. 
Finally, we use Kostant-Rosenlicht theorem which says that unipotent groups have closed orbits (\cite[Proposition 4.10]{Borel}). 	
\end{proof}

\subsubsection{$\mathcal P_l$ is somewhere a trivial fibration}

	Let $\widehat{U} $ be the (open) subset of $\widehat{M}$ where $\hat{f}= \kappa_l^{\widehat{V}}$  has a maximal rank.  Consider 
	$\hat{f}(\widehat{M} - \widehat{U})  $ as a subset of   $\mathcal Y_l^{\widehat{V}}$.
	Since $\hat{f}$ is $H$-equivariant and  $M$ is compact, $\hat{f}(\widehat{M} - \widehat{U})  $ can be written as the $H$-saturation of 
	a compact subset $K \subset \mathcal Y_l^{\widehat{V}}$. 
	
	Now $\mathcal Y_l^{\widehat{V}}$ admits a $H$-equivariant stratification $\mathcal Y_l^{\widehat{V}}= Y_0 
	\supset Y_1 \supset  \ldots \supset Y_k$, such  that the $H$-action of each
	$Z_i= Y_i - Y_{i+1}$ defines a fibration.  Let $m$ be the  smallest integer such that 
	$A= \hat{f}(\widehat{U}) \cap Z_m \neq \emptyset$.  In particular $\hat{f}^{-1}(A)$ is open in $ \widehat{U}$. \\
	Let $K^\prime = K \cap Z_m$. 
	This is closed in $Z_m$
	and 
	hence
	its saturation $B= H.K $ is also closed in $Z_m$, since the $H$-orbits determine a fibration. By Sard Theorem, $B$ has 
	a vanishing Hausdorff measure of exponent the maximal rank of $\hat{f}$. In particular $B\cap A$ has empty interior 
	in $A$, and thus $\widehat{U}^\prime = \hat{f}^{-1}(A-B)$ is open,  and is mapped by $f$ into the regular values set of $\hat{f}$. 
	
	Now, saturate everything by the $H$-action; that is, instead of considering levels $\hat{f}^{-1}(z)$, consider 
	inverse images $\hat{f}^{-1}(H z)$. Since $H z $ is closed in $\mathcal Y_l^{\widehat{V}}$, $\hat{f}^{-1}(H z)$ projects onto 
	closed submanifolds in $M$.
	
	Project everything on $M$ and get an open set $U^\prime \subset M$ with a submersion 
	$f: U^\prime \to Z_m/H$. 
	Thus, on  $ {U}^\prime$, we have a partition into closed (in $M$) submanifolds (not necessarily connected) given 
	by the levels of a global submersion.  
	
	Let $\mathcal K$ be the so-defined foliation, i.e.  with  leaves   the connected components of the $f$-levels. 
	Let $\tau$ be a small transversal to a leaf $C$, so that $f_{\vert \tau}$ is injective. Let   $U^{\prime\prime}$ be the saturation of $\tau$ 
	by $\K$.  Now,  $f_{\vert U^{\prime\prime}}$ is a submersion with connected levels.  We already know that these levels are compact. 
	It turns out that this implies  that $f_{\vert U^{\prime\prime}}: U^{\prime\prime}\to f (U^{\prime\prime})$ is a bundle map (see \cite{bundle}).  In conclusion, we have  proved:
	
\begin{fact}  \label{fibration} There is an open subset $B \subset M$ where the $\mathcal P_l$-classes are  submanifolds of $M$, closed in $M$, and form a trivial fibration. 
\end{fact}

\subsection{Reduction} Let us now take  $ l =  \dim \O(1, n)$ and denote $\mathcal P_l$ simply by $\mathcal P$ (actually it is enough to take $l = \dim H = \dim \SO(n-1) + (n-1)$). From \cite{Pec}, we have:
	
\begin{proposition} \cite{Pec}  \label{integrability} On the open set $B \subset M$  defined in Fact \ref{fibration}, the $\G^0$-orbits coincide with the fibers. In particular each fiber is locally homogeneous (and closed in $M$).  
Furthermore, a finite index sub-pseudo-group of $\G$ preserves $B$ and has  the same orbits as $\G^0$.
\end{proposition}
	
\subsubsection{Existence of a core $N$, Proof of Theorem \ref{reduction}}\label{Par. Existence of a core N, Proof of Theorem 1.3}  If some fiber in Proposition \ref{integrability} is somewhere non-degenerate, i.e. the restriction of the metric to it is of Lorentzian type, then it is everywhere non-degenerate by local-homogeneity. Take $N$ to be this fiber. If all fibers are degenerate, then consider a transversal (to the fibration) curve $c$, and let  $N$  be its saturation (by the fibration).  So $N$ is a Lorentzian manifold with boundary, diffeomorphic to a product $F \times [0, 1]$, where $F$ is closed in $M$. The  local isometry  pseudo-group $\G^0$  preserves and acts transitively on  the lightlike geodesic factors $F\times \{u\}$.

\section{Further results on the core $N$}\label{Further results}

Our goal is to study the $V$-dynamics on $M$, by first replacing $M$ by $M^*$ (as introduced in Par. \ref{Subsection: M*}), and then replacing $M$ (which is in fact $M^*$) by   its core $N$. This section is devoted to justifying these reductions and proving the    global algebraic structure of  the core $N$ stated in the second part of Theorem \ref{reduction}.

\subsection{Passing  to $M^*$ and getting  a unipotent isotropy}  \label{replace_M*} Consider $M^*$ the Stiefel manifold of orthonormal frames of $V^ \perp/ \R V$, as described in Section   \ref{Subsection: M*}.  We can do for it the same analysis as that for $M$ from the point of view of its pseudo-Lie algebra of local Killing fields. We will however here just lift the $\g$ -action to it (where $\g$ is the pseudo-Lie algebra of local Killing fields of $M$).

\begin{fact} \label{lift_M^*} Let $B^*$ be the inverse image of $B$  (defined in Proposition \ref{integrability}) in $M^*$ endowed with the local $\G^0$-action. Then 
	the $\G^0$ orbits in $B^*$ are closed submanifolds in $M^*$, and so they form a fibration.
\end{fact} 
\begin{proof}   For $x \in M$,  let $\G_x^0$
	denote  its stabilizer  in $\G^0$. Its (faithful) representation in 
	$T_xM$ makes it a subgroup  $I_x$  of $H = \SO(n-1) \ltimes \R^{n-1}$.  A standard fact from Gromov's rigid transformation
	groups  theory \cite{Gro},  or  its Cartan connections variant states that $I_x$ is algebraic. From this, one infers (see \cite[Theorem 4.3 of Chapter VIII]{hochschild}) that, up to a finite index, 
	$I_x$ has the form $K \ltimes \R^k$, where $K $ is a closed subgroup of $\SO(n-1)$.  Next, one can check that if $x^* \in M^*$
	projects on $x$, then the $\G^0. x^*$  fibers over $\G^0. x$ with fiber $K$. This implies in particular that 
	the $\G^0$ orbits on $M^*$ are closed. 
\end{proof}

The interest of working on $M^* $ lies in the following fact brought up in  Par. \ref{Subsection: M*}:

\begin{fact}   The $\G^0$-action on $M^*$  has unipotent isotropy.
	
\end{fact}
\begin{proof}
	The elements of the orthogonal group preserving a lightlike vector are given by $\SO(n-1) \ltimes \R^{n-1}$, and it follows from Corollary \ref{Corollary for unipotent isotropy} that the isotropy group of the $\G^0$-action on $M^*$ lies in the $\R^{n-1}$ factor, which is unipotent in $\O(1,n)$.
\end{proof}

\subsection{Equicontinuity on $N$ implies equicontinuity on $M$} \label{replace_N}

\begin{fact}  \label{criteria1} If the flow of $V$ acts equicontinuously on the core  $N$, then it acts equicontinuously on $M$
	
\end{fact}
\begin{proof} The fact is  actually true for any isometric flow $\phi^t$ acting on a compact Lorentzian manifold $(M, g)$:  if 
	it preserves  a closed  Lorentzian submanifold $N$ and acts equicontinuously on it, then it acts equicontinuously on $M$. Indeed, 
	the hypothesis means that $\phi^t$ preserves a Riemannian metric $h$ on $TN$. Construct a Riemannian metric 
	$h^\prime$ on $TM$ along $N$  given as $h^\prime = h \oplus g_{T^\perp N}$. Observe here that we used the fact that $g$ is spacelike on $T^\perp N$. Now, it is known that for flows preserving affine connections,   equicontinuity along a closed invariant subset  (not necessarily a submanifold) implies
	everywhere equicontinuity (see for instance \cite{Ze1}).
\end{proof}

\begin{fact}  \label{criteria2} In the local co-homogeneity one case,  if the flow of $V$ acts equicontinuously on a fiber of $N$, then it acts 
	equicontinuously on $N $ (and hence on $M)$.
	
\end{fact}

\begin{proof} Let $F$ be such a fiber. If $\phi^t$ preserves a Riemannian metric $h$, then it preserves 
	$E$ the orthogonal of $V$ with respect to $h$. This $E$ is spacelike (with respect to $g$). 
	Consider $E^\perp$, its orthogonal with respect to $g$.  It has dimension 2 and contains 
	$V$. There is a well defined lightlike vector field  $U$  in $E^\perp$ such that $g(U, V) = 1$. Since 
	$\phi^t$ preserves $V$, it preserves $U$ too.  Thus, $\phi^t$ acts equicontinuously on 
	$TN_{\vert F}$. As in the previous proof,  this implies that $\phi^t$ is equicontinuous on $N$.
\end{proof}

\subsection{Completeness  in  the universal cover $\widetilde{M}$} \label{global_action}

\subsubsection{Summarizing up}\label{summarizing}
Let  $(M,g,V)$ be a compact Brinkmann spacetime. In Par. \ref{Subsection: M*}, we associated to it a natural Brinkmann spacetime $(M^*,g^*,V^*)$ with a Lorentzian submersion $M^* \to M$, such that $V^*$ is the lift of $V$,  which we proved to be horizontal. Given the horizontality of  $V^*$, it is clear that the equicontinuity of $V^*$ implies the equicontinuity of $V$. Thus, we can reduce the proof of equicontinuity to that of $V^*$. 

On the other hand, for any Brinkmann spacetime, we constructed in Section \ref{Cartan_Connection} a core $N$ (contained in some open subset $B \subset M$), whose properties are outlined in the first part of Theorem \ref{reduction}. Using the fact that $N$ is a closed $V$-invariant submanifold of $M$, we proved in the previous paragraph that equicontinuity on $N$ implies equicontinuity on $M$. Next, we lifted $B$ to $M^*$, and lifted the $\mathcal{G}^0$-action to it. 
By doing so, we obtained in Fact \ref{lift_M^*} a $V^*$-invariant closed submanifold $N^*$ of $M^*$, which is either locally homogeneous or has local cohomogeneity one. In the local cohomogeneity one case, $N^*$ is diffeomorphic to a trivial fibration, whose fibers are lightlike, tangent to $V^*$, locally homogeneous and closed in $M^*$ 
(see Par. \ref{Par. Existence of a core N, Proof of Theorem 1.3}). Based on all these facts, it follows that the equicontinuity of $V^*$ on $N^*$ implies the equicontinuity of $V$ on $M$.
\medskip

All these developments justify replacing $M$ by $M^*$, and then $M^*$ by the core $N^*$. So, 
henceforth,  we will assume  that our initial    Brinkmann spacetime $(M, g, V)$  is such that:

\medskip

$(i)$ Either $M$ is locally homogeneous, or  has a  local co-homogeneity one type. In the local co-homogeneity one case, the leaves of $\F$ are assumed to be locally homogeneous and compact.

$(ii)$ The isotropy is unipotent.
\medskip

\noindent We have the following algebraic description:

\begin{proposition}   \label{algebraic.description}  Let $(M,g,V)$ be a compact Brinkmann spacetime satisfying property (i) above. Let $G$ be the identity component of the isometry group 
	$\Iso(\tilde{M}, \tilde{g}, \tilde{V})$ of the universal cover, endowed with the lift of the Brinkmann 
	structure. Then, 
	either $M$ has the form $ \Gamma \setminus G /I$, or it is a trivial fibration $ M \to [0, 1]$, 
	and each fiber has the form $ \Gamma \setminus G / I_u$, where $I_u$ depends on the parameter $u \in [0, 1]$.
\end{proposition} 
In particular, the core $N$ (constructed in Section \ref{Cartan_Connection}) of any compact Brinkmann spacetime,    is itself a compact Brinkmann spacetime that satisfies property (i). So this proposition yields the global algebraic structure of $N$ stated in the second part of Theorem \ref{reduction}.
\bigskip

The remaining part of the section is devoted to the proof of Proposition \ref{algebraic.description}. Here, we also assume that $M$ satisfies (ii), allowing a more precise algebraic description.
\medskip

Remember the notations of the objects in  $(M, g)$: $V$ the Brinkmann         parallel  vector field,  
$\mathcal V$ the $1$-dimensional foliation that it determines, $\F$ the codimension $1$ lightlike geodesic foliation 
tangent to $V^\perp$.  
We denote  by 
$\tilde{V}$, $\tilde{\mathcal V}$ , $\tilde{\F}$ the  lift of the corresponding    objects to  $\tilde{M}$.
If $F$ is a leaf 
of $\F$, then 
$\tilde{F}$ is a connected component of its lift in
$\tilde{M}$.

The universal cover $\widetilde{M}$ is 
either locally homogeneous or topologically of the form $\widetilde{F} \times [0, 1]$, and each 
$\widetilde{F} \times \{t\}$ is locally homogeneous, and obviously simply connected. In  both  cases, 
the whole $\widetilde{M}$ or the factors  $\widetilde{F} \times \{t\}$ are  (real) analytic. A classical 
result \cite{Nom} says, in this case, that any locally defined vector field 
$X$ on an open subset $U$ of $M$, extends coherently to $M$. So, the (global)  Killing algebra
$\g$ of $\widetilde{M}$ acts either transitively or with codimension one orbits.

Let us prove that this infinitesimal $\g$-action is complete, i.e. that any Killing field $X \in \g$ has a complete flow, or alternatively, 
that the  simply connected ``abstract'' group $G$ with Lie algebra $\g$ acts (globally) on $\widetilde{M}$.
We start treating the co-homogeneity one case, and show afterwards how to adapt the proof to the 
homogeneous case.

\subsubsection{The co-homogeneity one case}
\label{cohomogenityone}

\begin{fact} Let $\tilde{p} \in \tilde{F}$ and $\i \subset \g$ its stabilizer algebra.  Then, 
	$\i$ is abelian and acts unipotently on $T_{\tilde{p}} M$.
\end{fact}
\begin{proof}  The group $I$ determined by $\i$ is contained in the subgroup 
	of the orthogonal group 
	$(T_{\tilde{p}} \tilde{M}, \tilde{g}_{\tilde{p}})$ preserving $\tilde{V}(\tilde{p})$, and acting 
	trivially on $ \tilde{V}(\tilde{p})^\perp / \R \tilde{V}(\tilde{p})$ since the isotropy is unipotent. 
	This group is abelian and acts unipotently.
\end{proof}

\begin{fact} Let $\j$ be the Lie subalgebra preserving the $1$-leaf $\tilde{ \mathcal V} (\tilde{p})$. Then $\j = 
	\i \oplus \z$, where $\z = \R \tilde{V}$ (which is contained in the center of $\g$). Furthermore, 
	$\j$ is an ideal of $\g$.  
\end{fact}
\begin{proof} $\j$ contains $\i \oplus \z$ and $\i$ has codimension  $1$ in $\j$, so we have equality. 
	The Lie algebra $\g$ acts, locally, on the quotient space $\tilde{F}/\tilde{V}$ by preserving 
	a parallelism. The Lie algebra $\j$ acts by fixing  the point $\tilde{V} (\tilde{p})$ 
	of 
	$\tilde{F}/\tilde{V}$ and hence acts trivially on it 
	(an automorphism of a parallelism is trivial if it has a fixed point). 
	It follows that $\j$ is exactly the kernel of the $\g$-action on $\tilde{F}/\tilde{V}$, and hence 
	it is normal.
\end{proof}

\begin{corollary} $\j$ is contained in the nil-radical of $\g$.
	
\end{corollary}  
\begin{proof} 
	Because $\j$ is an abelian ideal.
\end{proof} 

\begin{fact} Let $G$ be the simply connected Lie group determined  by $\g$, $I$ and $J$ 
	its subgroups tangent to $\i$ and $\j$ respectively. Then $I$ and $J$ are closed in $G$
	(and are simply connected). Furthermore, $\tilde{F} $ is isomorphic to $G/I$. 
	
\end{fact}

\begin{proof} Since $\j$ is an ideal in $\g$, $J$ is a normal Lie subgroup of the simply connected Lie group $G$, hence closed in $G$ (see for example \cite[p. 610]{Mostow50}), and simply connected (see \cite[Lemma 1]{Mostow50}). Moreover,  $J$ is abelian and simply connected, so the subgroup $I$ (which is simply connected) is closed in $J$, and hence in $G$. 
By Theorem \ref{completeness} (or even its partial version in Section  \ref{Section: Completeness along the leaves}),  the leaf $F$ is complete, and thus $\tilde{F}$ is complete. Consequently, the (global) Killing fields tangent to $F$ are complete (\cite[Ch.9 Proposition 30]{O'Neill}), so, by Palais theorem \cite{palais}, the infinitesimal isometric action of the Lie algebra $\g$ integrates into a transitive Lie group action of $G$. Since $I$ is a closed subgroup of $G$, one can consider the homogeneous space $G/I$, and $\tilde{F}$ is locally modeled on it (see \cite{Mostow50}). Therefore, we have a developing map $d: \tilde{F} \to G/I$, equivariant for the left action of $G$, and which is then a diffeomorphism.  
\end{proof}

\subsubsection{The locally homogeneous case}  Now, in the locally homogeneous case, consider $\h$ the codimension $1$ subalgebra  of $\g$ preserving a leaf 
$\widetilde{F}$ of $\widetilde{\mathcal F}$. The space of leaves $\widetilde{\F}$ is defined by a closed $1$-form $\omega=\widetilde{g}(\widetilde{V},\cdot)$, so it has an affine structure preserved by the action of $\g$ (see \cite[Section 3]{Content1} for more details). A (local) isometry preserving $\widetilde{V}$ preserves $\omega$, hence induces a (local) translation on the (affine) space of leaves.
Therefore, $\h$ preserves in fact individually these leaves, and is thus an ideal. 
Its associated subgroup $H$ in $G$ is closed, and as in the previous case, the isotropy group 
is closed in $H$ and hence in $G$. As in the previous case, we can take   the quotient $G/I$ and 
use completeness of $M$  to deduce that $\widetilde{M}$ is isomorphic to $G/I$. $\Box$.

\section{The degenerate case, the $\overline{\mathcal V}$-foliation on a  leaf $F$}\label{Section 8}

\label{closure foliation}
We  investigate  here the degenerate case, that is, when $M$    has  local co-homogeneity one, and $\F$ has moreover all its leaves compact (see Par. \ref{summarizing}).
\medskip

In all the section, we fix a leaf $F$ of $\F$.

\subsection{The cocktail of geometries on $F$}\label{Par. 8.1} We will   
exploit existence of many compatible geometric structures on the compact manifold $F$, which we list here:

\begin{itemize}

\item
A lightlike metric induced from $(M, g)$. 

\item 
A connection induced from $M$,  since $F$  is geodesic.  

\item

The  parallel vector field $V$.

\item

The $1$-foliation $\mathcal V$ determined by $V$.  It is  transversally Riemannian, as it is the case of the characteristic (null) foliation 
on any lightlike geodesic submanifold in a Lorentzian manifold (see for instance \cite{Zeg_geodesic}). 
\item

The (pseudo) Lie algebra $\g$ acting transitively on  $F$ by preserving all these structures.

\item

It is a standard technique in transversally Riemannian foliation theory to lift   the foliation to the Stiefel space of  
transversal orthonormal frames, in order to get a transversally parallelizable foliation. 
In our situation, we already replaced $M$ by  $M^*$, without loosing its Brinkmann nature, essentially for this aim (par.  \ref{summarizing}). So the $\mathcal V$-foliation on  $F$ is  transversally parallelizable. 
\end{itemize}

\subsection{Dynamics of the closure foliation $(F, \overline{\mathcal V})$,     the holonomy subgroup $\Gamma^0$.}

One fundamental result in Riemannian foliation theory is that taking the closure of a Riemannian foliation $\mathcal V$, gives
rise to a singular Riemannian foliation $\overline{\mathcal V}$, which is actually  regular in the transversally parallelizable case \cite{Mol}. 
Carri\`ere's Theorem,  \cite{Car1, Car2}, applies  since  $\dim \mathcal V = 1$,  and  says that the $\overline{\mathcal V}$-leaves are tori. 
The goal of this section is to show in our rich situation that,  essentially,   $\overline{\mathcal V}$ is given by an isometric  action of a torus. 

\subsubsection{Notations}  Fix  $p  \in F \in M $,  so  $F$ is the $  \F(p)$- leaf. Consider a lift  $\tilde{p} \in \tilde{M}$ of $p$ and let $\tilde{F}= \tilde{\F} (\tilde{p})$.

Consider now the following objects:

\begin{itemize}

\item
The foliation $\bar{\mathcal V}$ of $F$ given by the closure
of the $\mathcal V$-leaves, and let 
$\widetilde{\bar{ \mathcal V}} $ be its lift to $\tilde{F}$. 

\item
$I$,  the stabilizer of $\tilde{p} $ in $G$, then  $\tilde{F}   \sim G / I$.

\item   $J$, 
the stabilizer   of  $\tilde{\mathcal V} (\tilde{p})$ in $G$.  
Remember  that $J$ is normal in $G$ and is furthermore abelian and  contained in the nil-radical of $G$
(par. \ref{cohomogenityone}).

\item  $\pi: G \to Q = G/J$,  the quotient map. 
In fact, $Q$ is identified to the quotient space 
$\tilde{F} / \tilde{\mathcal V}$.

\item  $\Gamma = \pi_1(F) \subset G$. This also equals $\pi_1(M)$ since $M$ is a product of $F$ by an interval 
(Proposition \ref{algebraic.description}). 

\item $L$, the closure in $Q$ of $\pi(\Gamma)$, and $L^0$ its identity component. 

Set  $P = \pi^{-1}(L)$ and $P^0 = \pi^{-1}(L^0)$. Then $P^0$ is the identity component of $P$, and it contains $J$. Furthermore, $\pi(\Gamma) \cap L^0$ is dense in $L^0$.

\item  $\Gamma^0 = \Gamma \cap P^0$. As $L^0$ is normal in $L$, $P^0$  and $\Gamma^0$ are  normal in $P$ and 
$\Gamma$ respectively.

\end{itemize}

\subsubsection{Algebraic description.} Let us give an algebraic description, i.e. by means of $G$, of all the present quantities:

\begin{itemize}

\item

First, $\tilde{F} = G/I$, 

\item  For $x \in G$, the $\tilde{\mathcal V}$-leaf of $xI$ is $Z x I $, where $Z$ is the simply connected Lie group determined by $\z = \R \tilde{V}$, and we have $ Z x I = x Z I = x J = J x$. So it corresponds to the subset $J x$ of $G$.

\item The quotient leaf-space $\tilde{F} / \tilde{\mathcal V} = G/J = Q$.

\item  The orbit  $P (xI)$ corresponds to the subset of $G$:  $P x I = P (x I x^{-1}) x = P.x $, since $x I x^{-1} \subset 
x J x^{-1} = J \subset P$. 

\item  Similarly $P^0 x I = P^0 x \subset G$.

\item The $\Gamma$-saturation of a leaf $\tilde{\mathcal V} (x I)$  corresponds to 
$\Gamma J x I  = \Gamma J x$, and its closure is $\overline{\Gamma J x} = \overline{\Gamma J } x = P x $.

\item The connected component of  $xI$ in    $ \overline{\Gamma   \tilde{\mathcal V} (x I)}$ corresponds to 
$P^0 x
= \overline {\Gamma^0 J}x $.  In other words, the connected component of the lift in $\tilde{F}$ of the closure of any $\mathcal V$-leaf in $F$, has the form 
$P^0x$ (as a subset in $G$ or as a $P$-orbit of $xI \in G/I$).  Therefore,  
the closure foliation $\overline{\tilde{\mathcal V}}$ corresponds to the $P^0$-action.

\end{itemize}

\subsubsection{Structure of $L^0$}

\bigskip

Observe that: $L^0 = 1$ $\iff$ $P^0 = J$ $\iff$
$\Gamma \cap J \neq 1$, and all these are  equivalent to  the fact that all $\mathcal V$ 
leaves in $F$ are periodic, with the same period.  This periodic case is trivial with respect to our considerations here and so 
henceforth, we will assume $L^0 \neq 1$.

\begin{fact} $\Gamma^0$ is abelian. Furthermore,  assume we are not in the periodic case, then $\pi(\Gamma^0)$ is dense in $L^0$, 
and  $\pi:  P^0 \to L^0$  is injective on $\Gamma^0$,  and hence $L^0$ is abelian.

\end{fact}

\begin{proof} This follows from Carri\`ere's Theorem \cite{Car1, Car2} on closure of orbits of transversally Riemannian flows
(i.e. foliations of dimension $1$).  With our notation, it says that a  leaf of $\overline{\mathcal V}$ 
is a torus, on which $\mathcal V$ is diffeomorphic to a minimal (i.e. having dense leaves) linear foliation (of dimension $1$).
We are here in a transversally parallelizable case, where all  $\overline{\mathcal V}$- leaves
are diffeomorphic. They are all given by $P^0$-orbits. Such an orbit is of the form $ \Gamma^0 \setminus P^0/I$. 
Since the $\overline{\mathcal V}$-leaf is a torus, and $\Gamma^0$ is the deck group of its cover $P^0/I$, 
then $\Gamma^0$ is abelian. Now, $L^0$ is the quotient space of the $\mathcal V$-foliation when lifted to the universal cover.
Since $\mathcal V$ is minimal (in the $\overline{\mathcal V}$-leaf), then $\pi(\Gamma^0)$ is dense in $L^0$.
Finally,  besides the periodic case, if we assume $\Gamma^0 \cap J = 1$, then $\Gamma^0$ injects in $L^0$, and $L^0$ is therefore abelian.
\end{proof}

\begin{remark}  
The quotient space 
$\tilde{F} / \tilde{\mathcal V}$ is identified to the quotient group  $ Q = G/J $.  Using the  language of transversally Riemannian foliation \cite{Mol},  one says that $\mathcal V$ is a transversally Lie foliation with structural group $Q$,  which means that $\mathcal V$ has a transversally    geometric structure modeled on $(Q, Q)$ (where the group $Q$ acts by left translation on the space $Q$). In general, one reduces 
the study of   transversally Riemannian foliations, first to  the transversally  parallelizable case. Then, 
for closures of leaves, or say when there is  a dense leaf, one proves  that the transversally parallelizable foliation is indeed    a transversally Lie foliation.  Actually, Carri\`ere's Theorem  \cite{Car1, Car2} was proved for transversally Lie foliations. Here, we  have   a richer situation, where the ambient manifold itself  $F$  has a  (local) geometric structure of type 
$(G, G/I) $,  which  induces a transversally Lie foliation  structure of type $(Q, Q)$ for $\mathcal V$.
\end{remark}

\subsection{The syndetic hull $H(\Gamma^0, P^0)$}

\begin{proposition} \label{syndetic_hull}

Let $C^0$ be the identity component of the centralizer $Z(\Gamma^0)$ of $\Gamma^0$ in $P^0$. 
Then:

- $C^0$ is transversal to $J$  (in $P^0$) and hence acts transitively on the $\overline{\tilde{\mathcal V}}$-leaves,

- $C^0$ is abelian and contains $\Gamma^0 $.
\end{proposition}

\begin{definition}[Syndetic hull] As $C^0$ is a connected abelian Lie group and 
$\Gamma^0 \subset C^0$ is discrete, there exists a unique subgroup 
of $C^0$  containing $\Gamma^0$ and in which 
$\Gamma^0$ is a lattice. Since it  was defined by means of $P^0$, it will be 
denoted by $H(\Gamma^0, P^0)$ and called the syndetic hull of $\Gamma^0$ 
in $P^0$.
\end{definition}

\begin{remark}  
In the nilpotent case, there is a construction of a (unique)  Malcev envelope  which      associates to a discrete 
subgroup (say of finite type) a syndetic hull where it is a lattice \cite{Rag}. In our case here, 
$P^0$ is solvable,  a situation where the construction does not extend. Actually, 
contrary to ``easy  life'' in  semi-simple or nilpotent groups, 
in solvable groups and worse, in semi-direct products of type compact by solvable, 
there is no way to fill in discrete sub-groups, 
even Borel density fails for    lattices 
(compare with  \cite{Witte})!
\end{remark}

\begin{proof}
We have an exact sequence $1 \to J \to P^0 \to L^0 \to 1$. Since $J$ is abelian, the $P^0$-action by conjugacy on $J$ reduces to a $L^0$-action (on $J$). For $l \in L^0$, let $Z_J(l)$ be  the fixed point set of its action on $J$. The trivial factor 
of the $L^0$ representation is thus $T= \cap_{l \in L^0} Z_J(l)$.  So $T$ is the intersection of 
$J$ with the center of $P^0$. 
For a generic $l$, $Z_J(l) = T$. More precisely, there is a finite set of (proper and closed) subgroups
$L_1^0, \ldots, L_k^0$, such that if $l \notin \cup_i L_i^0$, then $Z_J(l) = T$.
Similarly, one defines $Z_J(p)$ for $p \in P^0$, specially 
for $p = \gamma \in \Gamma^0$.  Since $\pi(\Gamma^0)$ is dense in $L^0$, for 
$\gamma$ ``generic'', $Z_J(\gamma) = T$. 
By taking the quotient $P^0 / T$, we can assume $T = 1$ (observe, however, that  by doing this, i.e. passing to the quotient space, $\Gamma^0$
could become non-discrete).

\medskip

Let $\gamma$ be generic and assume it belongs to two one parameter groups $h_1(t)$ and $h_2(t)$ having the same projection in $L^0$, that is $\pi(h_1(t)) = \pi(h_2(t))$.  On the one hand,  we have $\pi(h_1(t) h_2^{-1}(t)) = 1$, and thus  $h_1(t) h_2^{-1}(t)$ belongs  to $J$. On the other hand, $h_1(t) h_2^{-1}(t)$ commutes with $\gamma$. This implies $h_1(t) = h_2(t)$. 
Consider now  $c \in P^0$  commuting with $\gamma$. For any one parameter group $h_1(t)$ containing $\gamma$, 
$h_2(t) = Ad_c (h_1(t))$ is another one having the same projection in $L^0$ (since $L^0$ is
abelian, and hence the $Ad_c$ -action on it  is trivial).  This implies that $h_1 = h_2$, that is,  $c$ commutes with any  element   $h_1(t)$.

From all this, we infer that if  $\gamma$ and $\gamma^\prime$ are  two generic elements of $\Gamma^0$, then any one parameter groups 
$h(t) $ and $h^\prime(t)$ containing $\gamma$ and $\gamma^\prime$, respectively, commute. 
Regarding existence of such one parameter groups $h(t)$, 
let $K$ be a one parameter group in 
(the abelian) $L^0$ containing $\pi (\gamma)$,  and consider  $\pi^{-1}(K)$, which is  connected 
(since $J$ is connected), and 
contains $\gamma$. Then $\pi^{-1}(K)$ is a semi-direct product $\R \ltimes \R^k$ ($L^0 \cong \R^k$).  Let us refer to Lemma \ref{existence_one_parameter}  below, for the proof that generic elements of such a semi-direct product can be reached by  one parameter groups. 
\medskip

Now,  let  $\gamma_1, \ldots, \gamma_d$ be generic  elements of $\Gamma^0$ contained in one parameter groups
$h_1(t),.\\..,h_d(t)$, such that the derivatives $h_1^\prime(0), \ldots, h_d^\prime (0)$ generate
a subspace of the Lie algebra of $P^0$ of dimension 
$d = \dim L^0$. They determine an abelian group $A$ transversal to $J$ and centralizing $\Gamma^0$.  
Let us prove that
$A$ contains $\Gamma^0$.  To be precise on the significance of  genericity of  elements of $\Gamma^0$, 
let us denote by $U$ the open dense set  $ L^0 -  \cup_i L_i^0$, where the $L_i^0$ are  the subgroups defined above (where   the centralizer does not achieve the minimal dimension). Then,  $\gamma$ is generic if $\pi(\gamma) \in U$. 
Therefore,    $B = \pi^{-1}(U) \cap \Gamma^0  $  is contained in $A$.  \\
The subgroup $D$ generated 
by  $B$ is contained in $A \cap \Gamma^0$ and the projection  $\pi$ sends  $D$ to  the subgroup
generated by 
$\pi(\Gamma^0) \cap U$.  
But the last subgroup equals $\pi(\Gamma^0)$, since 
$\pi(\Gamma^0)$ is dense (in $ L^0)$ and $U$ is open (say, if $X$ is a dense subgroup of $\R^d$ and $U$ is open in $\R^d$, then for 
$a, b \in X \cap U$ close, the translated $X \cap U - b$ is a neighbourhood of $0 $ in $X$, and hence $X \cap U$ generates $X$).  
It follows that $\pi $ sends $D$ surjectively 
onto $\pi (\Gamma^0)$. But $\pi$ maps bijectively $\Gamma^0$ to $\pi(\Gamma^0)$, and $D \subset \Gamma^0$ which implies that   $D = \Gamma^0$, that is $\Gamma^0 \subset A$.

Now, if the trivial factor $T$ was not trivial, then $C^0$, the identity component of the centralizer
of $\Gamma^0$ is exactly $A T$.  It satisfies all the claimed properties.
\end{proof}

In order to finish the proof, we need:

\begin{lemma} \label{existence_one_parameter} Let $K$  be a semi-direct product 
$ \R \ltimes \R^d$. Then, any generic element of \textcolor{red}{$K$} belongs to a one parameter group. More precisely, assume
that $\R$ acts on $\R^d$ via a representation 
$ t \to e^{t a}$ with  $a \in \mathfrak{gl}(d, \R)$.
Let $\lambda_1, \ldots, \lambda_l$ be the purely imaginary eigenvalues of  $a$. If $t \notin \cup _i        \frac{ 2 \pi}{\lambda_i}  \sqrt{-1}\Z $, then, for any $v \in \R^d$,    $(t, v)$ belongs to some one parameter group of $K$.  
\end{lemma}

\begin{proof} 
The  Lie algebra $\mathfrak{k}$ of $K$, 
is generated by one element $Y$ and $\R^d$, with non-vanishing brackets 
$[Y, u ]= a(u)$, for $u \in \R^d$. It embeds in $\mathfrak{aff}(\R^d) = 
\mathfrak{gl}(d, \R)  \ltimes  \R^d$, the Lie algebra of the affine group
$\Aff(\R^d)$, by sending $Y $ to $a$ and $\R^d$ to $\R^d$. Let
$K^\prime$  be the  Lie subgroup of $\Aff(\R^d)$ determined by this embedding. Then, $K$ is isomorphic to $K^\prime$, unless $e^{ta}$ is periodic:  there exists $t_0  \neq 0$ such that $e^{t_0 a} = 1$.   In this case,
$K$ will be the (cyclic) universal cover of $K^\prime$. Such $K$ is a non-linear group, i.e. cannot be injectively 
embedded in any $\GL(n, \R)$. Their prototype is the Euclid group 
$\mathsf{Euc}_2 = \SO(2) \ltimes \R^2$. Actually, once treating $K^\prime$, we can lift to $K$, and hence we will 
assume we are not in this   periodic case.

Consider now the map  $\phi: (t, x) \in  \R \times \R^d \to \phi^t( x) = e^{ta}x + (e^{ta}u - u) \in \R^d$. 
One checks that $t \to \phi^t $  is a one parameter group 
in $\Aff(\R^d)$. Its infinitesimal generator is the vector field $x \to X(a) = a (x) + a(u)$. Accordingly, 
the vector field $x \to a(x) + u$ has a flow 
$(t, x)  \to  e^{ta}x + a^{-1}(e^{ta}u - u) $, where 
here $a^{-1}(e^{ta} -1) $ is understood to be equal to  $B_a(t) =  1 + \frac{t^2}{2!}a +  \frac{t^3}{3!}a^2 + \ldots $.
Let us see when $B_a(t)$ is not surjective. By means of a  Jordan decomposition, we reduce the question to the case where $a$ has a unique (complex) eigenvalue $\lambda$, that is $(a-\lambda 1) ^d = 0$. 
First, if $\lambda = 0$, then $a$ is nilpotent and $B_a(t)$ is surjective for any $t$. 
If $\lambda \neq 0$, then $a^{-1}$ exists and $B_a(t) = a^{-1}(e^{ta} -1)$. Thus, if $\det\, B_a(t) = 0$, then 
$e^{t\lambda}= 1$;  in particular,  $\lambda \in  \sqrt{-1} \R  $ and $t \in    \frac{ 2 \pi}{\lambda}  \sqrt{-1}\Z$ .
In sum, let $\lambda_1, \ldots, \lambda_l$ be the purely imaginary eigenvalues of  $a$. If $t \notin \cup _i  \frac{ 2 \pi}{\lambda_i}  \sqrt{-1}\Z $, then, for any $v \in \R^d$,    $(t, v)$ belongs to some one parameter group. 
\end{proof}

\subsection{$\Gamma^0$ seen  in the identity component of its centralizer  in $G$}

\begin{proposition}  \label{C^000} Remember the definition of $C^0$ (in Proposition 
\ref{syndetic_hull}) and that of the syndetic hull $H(\Gamma^0, P^0)$. 
Define the following subgroups of $G$:

- $C^{00}$: the identity component of the centralizer of $\Gamma^0$ in $G$.

- $C^{000}$: the identity component of the  center of $C^{00}$.

Then $C^{000}$ contains $H(\Gamma^0, P^0)$ and in particular acts transitively on the $\overline{\tilde{\mathcal V}}$-leaves. 

In fact, $H (\Gamma^0, P^0)$ is  also the syndetic hull of $\Gamma^0$ in $C^{000}$.

\end{proposition}

\begin{proof} 	
	If $g \in G$ centralizes $\Gamma^0$, then it normalizes $J \Gamma^0$ (since $J$ is normal) and hence 
	normalizes $P^0 = \overline{J \Gamma^0}$.  It follows that the centralizer of $\Gamma^{0}$, and hence 
	also its identity component $C^{00}$, normalizes $P^0$. But then, $C^{00}$ preserves $C^0$ too,  since 
	the latter is defined as the identity component of $Z (\Gamma^0)$ in $P^0$. It also preserves 
	the syndetic hull $H(\Gamma^0, P^0)$. But, since it centralizes the lattice 
	$\Gamma^0$ in $H(\Gamma^0, P^0)$, $C^{00} $   acts trivially on $H(\Gamma^0, P^0)$.  That is 
	$H(\Gamma^0, P^0) \subset C^{000}$, and hence $H(\Gamma^0, P^0) = H(\Gamma^0, C^{000})$.
\end{proof}

\section{End of the  Proof of Theorem \ref{dynamics}}
\label{end of proof}

So far, we fixed a leaf $F$ and considered its $\mathcal V$ and $\overline{\mathcal V}$-foliations. Now, we consider 
a small  curve $p: u \in [0, 1]  \to p_u$ transversal to $\mathcal F$.   This leaf $F_u = \mathcal F (p_u)$ comes with its
associated foliation $\overline{\mathcal V}$. There is, a priori, no obvious continuity or even semi-continuity 
of  $\overline{\mathcal V}$. In other words,  $\overline{\mathcal V}$ is a foliation on each leaf but not a 
foliation of $M$. 

We  have in particular groups  $I_u$, the stabilizers of $\tilde{p}_u$, and 
$J_u$ the stabilizers of    $\tilde{\mathcal V}(\tilde{p}_u)$. These two groups depend continuously on $u$.

There is also
$P^0_u$ and  $\Gamma^0_u$ associated to $F_u$.  Again,  a priori, they satisfy   no  obvious continuity, or 
even semi-continuity   on $u$.

\begin{fact} Let $u_1, u_2$ be two parameters $\in [0, 1]$, $P^0_{u_1}, P^0_{u_2}$, $\Gamma^0_{u_1}$ and 
	$\Gamma^0_{u_2}$ their associated groups.  If $\Gamma^0_{u_1}= \Gamma^0_{u_2} = \Gamma^0$, 
	then $\Gamma^0$ has the same syndetic hull in $P^0_{u_1}$ and $P^0_{u_2}$.   In particular 
	this syndetic hull acts transitively on the $\overline{\tilde{\mathcal V}}$-leaves of both 
	$F_{u_1}$ and $F_{u_2}$.
	
\end{fact}

\begin{proof}  Apply Proposition \ref{C^000} and  observe that
	$C^{000}$ is the same for $u_1$ and $u_2$ (it is defined by means of $\Gamma^0$ only). 
\end{proof}

\begin{fact}  We can assume that for a dense set $D \subset [0, 1]$ of parameters $u$, $\Gamma_u^0$ is a constant 
	$\Gamma^0$.
\end{fact}

\begin{proof}  
	Let $B(e, n)$ be the ball of radius $n$ and centered at the neutral element of $\Gamma = \pi_1(M)$, 
	with respect to 
	a word metric given by  some generating set.  Let $X_n = \{u \in [0, 1] \;  /  \; \Gamma^0_u \cap B(e, n)$ generates $\Gamma_u^0 \}$.  For $n \in \N$ fixed, the subsets $\Gamma^0_u \cap B(e,n), u \in [0,1]$, of $B(e,n)$ are in finite number.
	By Baire's Theorem, there exists $n_0$ such that $\overline{X_{n_0}}$ contains a non-trivial interval of parameters. 
	The map $u  \in  [0, 1] \mapsto \Gamma_u^0 \cap B(e,n_0)$, from $[0, 1]$  into subsets  of $B(e, n_0)$,  has a finite image. So there 
	is a level whose closure contains an interval. We will assume it is $[0, 1]$ itself.
\end{proof}

\begin{corollary} All the syndetic hulls $H(\Gamma^0, P_u)$ coincide, for $u \in D \subset [0, 1]$  a dense set of parameters,  
	say $H(\Gamma^0, P_u) = H$,  for $u \in D$.  In particular $H$  acts transitively on $\overline { \tilde{\mathcal V}}$-leaves
	of all points $\tilde{p}_u$, $u \in [0, 1]$.	
\end{corollary}

\begin{proof} The previous fact implies  constancy   $H = H(\Gamma^0, P_u)$, for $u$ is a
	dense subset $D \subset [0, 1]$. 
	It then follows that 
	$H$ acts transitively  on a dense set of  $\overline { \tilde{\mathcal V}}$-leaves. This extends to all leaves.
\end{proof}

\begin{proof}[\textbf{End of the proof of  Theorem \ref{dynamics}}]
Consider the cover $M^\prime = \tilde{M} / \Gamma^0$. 
Then $H$ acts on it since it centralizes $\Gamma^0$. 
The $\F^\prime$-leaves of $M^\prime$  have the form $\Gamma^0 \setminus G / I$. As a class in $G$, a  point $\Gamma^0 x I$ has 
an $H$-orbit $H \Gamma^0 x I = H x I$. This  is  a torus $ \Gamma^0 \setminus H $ in $M^\prime$.  Taking $HZ$ (which is still abelian since $Z$ is contained in the center of $G$) instead of $H$ if necessary, we may assume that $Z \subset H$. 
Let $\tau$ be the image of the transversal 
curve $u \to p_u$  considered above,  and $\tau^\prime$ a lift in $M^\prime$.  The orbit $T^\prime= H \tau^\prime$
is topologically a product of a torus by an interval. It embeds to a  submanifold $T$ in $M$. It is Lorentzian 
and $V$-invariant, in fact $\overline{\mathcal V}$-invariant.  The $V$-action  on $T$ commutes with the 
$H$-one.  So, the $V$-action on each torus is conjugate to a linear one  and hence equicontinuous. 
With respect to Facts \ref{criteria1} and \ref{criteria2},  this $T$ can be thought of as the core $N$: equicontinuity 
of $V$ on a torus in $T $ implies equicontinuity on $T$, and  then equicontinuity on $M$ (since $T$ is timelike). \end{proof}

\begin{remark}
The case where the foliation $\F$ is not minimal, i.e. when the leaves are compact, is covered by the study we made so far. Indeed, in this case, the core $N$ contains a closed submanifold of the form $F \times [0,1]$, where $F$ is closed, $V$-invariant, and locally homogeneous, hence the results in Section \ref{Section 8} and Section \ref{end of proof} apply. 
\end{remark}

\begin{remark} Carri\`ere's Theorem says that $\mathcal V$ restricted to a $\overline {\mathcal V}$-leaf 
	is diffeomorphic to the foliation determined by a  minimal linear flow on the torus. In this general  context, the  
	transversally Riemannian foliation 
	$\mathcal V$ is given without  parametrization. All the investigation in last sections  aimed to 
	check that in our parametric setting, the vector field $V$  itself  is (smoothly) conjugate to a linear one. 
\end{remark}

\bibliographystyle{abbrv}
\bibliography{Bibliographie}

\vspace{1cm}

\noindent
{\small The first author was supported by the LABEX MILYON (ANR-10-LABX-0070) of Universit\'e de Lyon, within the program ``Investissements d'Avenir'' (ANR-11-IDEX-0007) operated by the French National Research Agency (ANR). And, partially, by the grants: PID2020-116126GB-I00\\
(MCIN/ AEI/10.13039/501100011033), and the framework IMAG/ Maria de Maeztu,\\ 
CEX2020-001105-MCIN/ AEI/ 10.13039/501100011033.}
\end{document}